\documentclass[final,1p,times]{elsarticle}

\usepackage{mathrsfs,amsmath,amssymb,bm,amsthm}
\usepackage{hyperref}
\usepackage{graphicx}
\usepackage{subfigure}
\usepackage{caption}
\usepackage{overpic}
\usepackage{epsfig}
\usepackage{xcolor}
\usepackage{indentfirst}
\usepackage{fancyhdr}
\usepackage{makecell, rotating}
\usepackage{algorithm} 
\usepackage{algpseudocode}
\usepackage{float}
\usepackage{fullpage}

\usepackage{xkeyval}
\usepackage{todonotes}
\presetkeys{todonotes}{inline}{} 
\usepackage[draft]{changes}
\usepackage{mathtools}
\mathtoolsset{showonlyrefs}
\usepackage{lipsum}
\definechangesauthor[name={Kai Jiang}, color=blue]{jk}

\newcommand{\bbR}{\mathbb{R}}
\newcommand{\bbC}{\mathbb{C}}
\newcommand{\bbZ}{\mathbb{Z}}

\newcommand{\bbN}{\mathbb{N}}

\newcommand{\ba}{\bm{a}}
\newcommand{\bb}{\bm{b}}

\newcommand{\br}{\bm{r}}
\newcommand{\bh}{\bm{h}}

\newcommand{\bN}{\bm{N}}
\newcommand{\bk}{\bm{k}}

\newcommand{\bB}{\mathbf{B}}
\newcommand{\bA}{\mathbf{A}}
\newcommand{\bI}{\mathbf{I}}

\newcommand{\calP}{\mathcal{P}}
\newcommand{\calR}{\mathcal{R}}

\newcommand{\hphi}{\hat{\phi}}
\newcommand{\hPhi}{\hat{\Phi}}

\newcommand{\vzero}{\mathbf{0}}
\newtheorem{thm}{Theorem}[section]
\newtheorem{lemma}[thm]{Lemma}
\newtheorem{proposition}[thm]{Proposition}

\newtheorem{defy}[thm]{Definition}
\newtheorem{remark}[thm]{Remark}

\newcommand\tbbint{{-\mkern -16mu\int}}

\newcommand\dbbint{{-\mkern -19mu\int}}

\newcommand\bbint{
{\mathchoice{\dbbint}{\tbbint}{\tbbint}{\tbbint}}
}

\DeclareMathOperator*{\argmin}{\mathrm{argmin}}

\newcommand{\Prox}{\mathrm{Prox}}
\newcommand{\GProx}{\mathrm{GProx}}

\begin{document}

\title{An efficient method for computing stationary states of phase field crystal models}

\author{Kai Jiang, Wei Si}
\address{
 School of Mathematics and Computational Science, Xiangtan
 University, Xiangtan, Hunan, P.R. China, 411105.
 \\
 }

\author{Chenglong Bao\corref{cor}}
\address{
 Yau Mathematical Sciences Center, Tsinghua University, Beijing, P. R. China, 100084.
 }
\cortext[cor]{Corresponding author. Email: clbao@mail.tsinghua.edu.cn}

\date{\today}

\begin{abstract}
	Computing stationary states is an important topic for phase
	field crystal (PFC) models. Great efforts have been made for
	energy dissipation of the numerical schemes when using
	gradient flows. However, it is always time-consuming due to
	the requirement of small effective time steps. In this paper, we
	propose an adaptive accelerated proximal gradient 
	method for finding the stationary states of PFC models. The
	energy dissipation is guaranteed and the convergence property
	is established for the discretized energy functional.
	Moreover, the connections between generalized proximal operator with classical (semi-)implicit and
	explicit schemes for gradient flow are given. Extensive
	numerical experiments, including two three dimensional
	periodic crystals in Landau-Brazovskii (LB) model and a two
	dimensional quasicrystal in Lifshitz-Petrich (LP) model,
	demonstrate that our approach has adaptive time steps which
	lead to significant acceleration over semi-implicit methods
	for computing complex structures. Furthermore, our result
	reveals a deep physical mechanism of the simple LB model via
	which the sigma phase is first discovered.

	\begin{description}
	\item[Keywords]
	Phase field crystal models, Stationary states, 
	Spectral collocation method,
	Semi-implicit scheme, Adaptive time step
	\item[AMS subject classifications] 35J60, 35Q74, 65N35
	\item[DOI]
	\end{description}
\end{abstract}


\maketitle


\section{Introduction}
\label{sec:intr}

The phase field crystal (PFC) model is an important approach to
describe many physical processes and material properties, 
such as the formation of ordered structures,
nucleation process, crystal growth, elastic and plastic
deformations of the lattice, dislocations,
etc\, \cite{chen2002phase, provatas2010phase}. 
More concretely, let the order parameter function be
$\phi(\br)$, the PFC model can be expressed by a free energy functional:
\begin{equation}
\begin{aligned}
	E[\phi(\br); \Theta] = G[\phi(\br); \Theta] + F[\phi(\br); \Theta],
	\label{eq:energy}
\end{aligned}
\end{equation}
where $\Theta$ are the physical parameters, $F[\phi]$ is the interaction
energy with polynomial type or log-type formulation and $G[\phi]$ is the
bulk energy that contains higher-order
linear operators to form ordered
structures \cite{brazovskii1975phase, lifshitz1997theoretical,
swift1977hydrodynamic}.
A typical interaction potential function for a bounded domain $\Omega$ is 
\begin{equation}
G[\phi] = \frac{1}{|\Omega|}\int_\Omega
\Big[\prod_{j=1}^{m}(\Delta+q_j^2)\phi \Big]^2 \,d\br, 
~~ m\in\mathbb{N}
\end{equation}
which can be used to describe the pattern formation of periodic crystals, quasicrystals and multi-polynary crystals.

In order to understand the theory of PFC models as well as
predict and guide experiments, it requires to find stationary 
states $\phi_s(\br; \Theta)$ and construct phase diagrams
of the energy functional \eqref{eq:energy}.
Mathematically, denote $V$ to be a feasible space,
one should solve the minimization problem
\begin{equation}
\begin{aligned}
\min_{\phi\in V} E[\phi(\br); \Theta],
	\label{eq:minEnergy}
\end{aligned}
\end{equation}
with different physical parameters $\Theta$, which brings
the tremendous computational burden.
Therefore, within appropriate spatial discretization, the goal of this paper is to develop an efficient
and robust numerical method for solving \eqref{eq:minEnergy} with
guaranteed convergence.

Most existing numerical methods for computing the stationary states
of PFC model can be classified into two categories. 
The first class of numerical methods solves the steady
nonlinear Euler-Lagrange equations of \eqref{eq:minEnergy}
through different spatial discretization approaches.
The second class of numerical methods has been designed via the
formulation of nonlinear gradient flow equations. 
In these numerical PDE approaches, the time-dependent nonlinear
gradient flows are discretized in space via different numerical methods.
In these time discretized approaches,
great efforts have been made to keep the energy dissipation 
which is crucial for convergence.
Typical energy stable schemes to the gradient flows include convex
splitting and stabilized factor methods, and recently developed
invariant energy quadrature, and scalar auxiliary variable
approaches for a modified energy\,\cite{shen2018scalar}. 
It is noted that the gradient flow approach is able to describe the
quasi-equilibrium behavior of PFC systems.  
Numerically, the gradient flow is discretized in both
space and time domain via different discretization techniques and the
stationary state is obtained with a proper choice of initial data.

Under an appropriate spatial
discretization scheme, the infinite dimensional problem
\eqref{eq:minEnergy} can be formulated as a 
minimization problem over a finite dimensional space.
Thus, there may exist alternative numerical methods that can
converge to the steady states quickly
by using modern optimization techniques. Similar ideas have
been shown success in computing steady states of 
the Bose-Einstein condensate \cite{wu2017regularized} and the calculation of density functional theory \cite{ulbrich2015proximal,liu2015analysis}.
In the PFC models, the discretized energy is nonlinear and
non-convex which consists of two parts:
bulk energy and interaction energy. Motivated by
the semi-implicit scheme and the accelerated proximal gradient
(APG) method \cite{beck2009fast,tseng2008accelerated} which has been successfully applied in image
processing and machine learning, we propose
an efficient numerical method for calculating the steady states of \eqref{eq:minEnergy}. 
As the traditional APG method is proposed for convex
problem and its oscillation phenomenon slows down the
convergence\,\cite{o2015adaptive,su2014differential}, 
the restart scheme has been used for accelerating the convergence. Moreover, the numerical
speed can be further accelerated by using the line search
starting with Barzilai-Borwein steps \cite{bao2018coherence}. 
The connection of classical explicit/implicit schemes
in gradient flows and proximal gradient methods is also built by 
defining a generalized proximal operator.
Extensive numerical experiments have demonstrated that our approach can
quickly reach the vicinity of an optimal solution with moderately
accuracy, even for very challenge cases.
As a byproduct, our numerical result reaveals a deep physical
intension of a simple PFC model, the Landau-Brazovskii (LB) model, by
obtaining the sigma phase.

The rest of this paper is organized as follows. Different
discretizations of the energy functional via the 
Fourier pseudospectral approach and the projection
method are introduced in section \ref{sec:discretization}. 
In section \ref{sec:method}, we present the gradient type method
and the adaptive APG method for solving the discretized
minimization problem.
The connection between our proposed approach and some existing
time discretized schemes in these numerical methods for solving
gradient-flow equations has been built in section
\ref{sec:connection}.  
Numerical results are reported in section \ref{sec:result} to
illustrate the efficiency and accuracy of our algorithms. 
Finally, some concluding remarks are given in section
\ref{sec:conclusion}.

\section{Physical models} 
\label{sec:model}

Two classes of PFC models are considered in the paper. The first
one is the Landau-Brazovskii (LB) model which describes periodic
structures\,\cite{brazovskii1975phase}.
The LB model was introduced to investigate the character of
phases and phase transition of periodic crystals. It has been discovered
in many different scientific fields, e.g., polymeric 
materials\,\cite{shi1996theory}. 
In particular, the energy functional of LB model is 
\begin{equation}
\begin{aligned}
	E_{LB}[\phi(\br)] = \frac{1}{|\Omega|}\int_\Omega
	\left\{\frac{\xi^2}{2}[(\Delta +1)\phi]^2 + \frac{\tau}{2!}\phi^2 
	-\frac{\gamma}{3!}\phi^3 + \frac{1}{4!}\phi^4 \right\}\,
	d \br,
	\label{eq:LB}
\end{aligned}
\end{equation}
where $\phi(\br)$ is a real-valued function which measures the
order of system in terms of order parameter.
$\Omega$ is the system volume, $\xi$ is the bare
correlation length, $\tau$ is the dimensionless reduced
temperature, $\gamma$ is phenomenological coefficient.  
Compared with double-well bulk energy,
the cubic term in the LB functional helps us study
the first-order phase transition. 

The second one is the Lifshitz-Petrich (LP) model that can
simulate quasiperiodic structures, such as the bi-frequency 
excited Faraday wave\,\cite{lifshitz1997theoretical}, and 
the explanation of the stability of soft-matter
quasicrystals\,\cite{lifshitz2007soft, jiang2015stability}.
Before we present the LP model, an introduction of the average
spacial integral, so-called almost periodic integral, is necessary.
For a space-filling structure, such as the quasicrystal, the
average spacial integral can be defined as
\begin{equation}
\begin{aligned}
	\bbint = \lim_{R\rightarrow \infty}\frac{1}{|B_R|}\int_{B_R},
	\label{eq:APintegral}
\end{aligned}
\end{equation}
where $B_R\subset\bbR^d$ is the ball centred at origin with radii $R$. 
Using the above notation, the energy functional of LP model is
given by 
\begin{equation}
\begin{aligned}
	E_{LP}[\phi(\br)] = \bbint
	\left\{\frac{c}{2}[(\Delta +q_1^2)(\Delta + q_2^2)\phi]^2 +
	\frac{\varepsilon}{2}\phi^2 
	-\frac{\kappa}{3}\phi^3 + \frac{1}{4}\phi^4 \right\}\,
	d \br,
	\label{eq:LP}
\end{aligned}
\end{equation}
$c$ is the energy penalty, $\varepsilon$ and $\kappa$ are
phenomenological coefficients.  

Formally, the difference between the LB and LP energy functional
is the number of length-scale governed by the differential term.
The LB model has a one-length-scale which can be used to
study the phase behavior of periodic
structures\,\cite{brazovskii1975phase, zhang2008efficient}, while the LP model 
possesses a two-length-scale that can be used to studied the
formation and stability of
quasicrystals\,\cite{lifshitz1997theoretical, lifshitz2007soft,
jiang2015stability, dotera2014mosaic}. 

\section{Discretization of the energy functional}
\label{sec:discretization}

In this section, we introduce different discretization schemes of the
energy functionals \eqref{eq:LB} and \eqref{eq:LP}, and reduce
them to finite dimensional minimization problems. 
Two classes of stationary states are considered. 
The first class of stationary states is periodic in LB model which 
can be described in a bounded domain. Thus we can truncate the energy
functional from the whole space $\bbR^d$ to a bounded domain
$\Omega$ with periodic boundary condition. 
Then we employ Fourier pseudospectral method to discretize LB
energy functional.
The second class of stationary phases can be quasicrystals in LP model. 
For these structures, the discretization of the energy functional 
in a bounded domain results in a significant Diophantine
approximation error. In this paper, we apply with the projection
method \cite{jiang2014numerical}, a high dimensional
interpretation approach, to discretize the LP energy function
\eqref{eq:LP}, which can avoid the Diophantine approximation error.

\subsection{Fourier pseudospectral discretization}
\label{subsec:fp}

Each of the $d$-dimensional periodic system can be described by a Bravis lattice 
\begin{equation}
\begin{aligned}
	\calR = \sum_{j=1}^d \ell_j \ba_j, ~~~ \ell_j\in\bbZ,
\end{aligned}
\end{equation}
where the vector $\ba_j\in\bbR^d$ forms the primitive Bravis lattice
$\bA=(\ba_1,\ba_2,\dots,\ba_d)\in\bbR^{d\times d}$.
The smallest possible periodicity, or named the unit cell, of the system is 
\begin{equation}
\begin{aligned}
	\Omega = \sum_{j=1}^d \zeta_j \ba_j,  ~~~ \zeta_j\in[0,1).
\end{aligned}
\end{equation}
The associated reciprocal lattice is 
\begin{equation}
\begin{aligned}
	\calR^* = \sum_{j=1}^d h_j \bb_j, ~~~ h_j\in\bbZ. 
\end{aligned}
\end{equation}
The primitive reciprocal lattice vector $\bb_j\in\bbR^d$ satisfies the
dual relationship
\begin{equation}
\begin{aligned}
	\ba_i \bb_j = 2\pi \delta_{ij}.
\end{aligned}
\end{equation}
Then the periodic function on the Bravis lattice, i.e.,
$\phi(\br)=\phi(\br+\calR)$, can be expanded as 	
\begin{equation}
\begin{aligned}
	\phi(\br) = \sum_{\bh\in\bbZ^d} \hphi(\bh) e^{i (\bB\bh)^T
	\br}, ~~~ \br\in\Omega,
	\label{eq:LBenergy:infinity}
\end{aligned}
\end{equation}
where $\bh=(h_1,h_2,\dots,h_d)^T$,
$\bB=(\bb_1,\bb_2,\dots,\bb_d)\in\bbR^{d\times d}$ is invertible. 
The coefficient,
$\hphi(\bh)=(1/|\Omega|)\int_{\Omega}\phi(\br)e^{-i(\bB\bh)^T
\br}\,d\br$, satisfies 
\begin{equation}
\begin{aligned}
	X := \left\{\{\hphi(\bh)\}_{\bh\in\bbZ^d}:
	\hphi(\bh)\in\mathbb{C}, ~
	\sum_{\bh\in\bbZ^d}|\hphi(\bh)|<\infty \right\}.
\end{aligned}
\end{equation}
In numerical computations, we need to minimize the LB energy
functional \eqref{eq:LB} in a finite dimensional subspace. 
More precisely, let
$\bN=(N_1+1, N_2+1, \dots, N_d+1)\in \bbN^d$, and 
\begin{equation}
\begin{aligned}
	X_{\bN} := \{\hphi(\bh)\in X:  \hphi(\bh) = 0, 
	~\mbox{for
	all}~ |h_j|> N_j/2, ~ j=1,2,\dots,d \}.  
\end{aligned}
\end{equation}
The number of elements in the set is $N=(N_1+1)(N_2+1) \cdots (N_d+1)$. 
The order parameter can be projected into the finite
dimensional space $X_{\bN}$, i.e.,
\begin{equation}
\begin{aligned}
	\phi(\br) \approx \sum_{\hphi(\bh)\in X_{\bN}} \hphi(\bh) e^{i (\bB\bh)^T
	\br}, ~~~ \br\in\Omega.
	\label{eq:LBenergy:finity}
\end{aligned}
\end{equation}
Due to the orthonormal condition 
\begin{equation}
\begin{aligned}
	\frac{1}{|\Omega|}\int_{\Omega} e^{i(\bB\bh_1)^T \br}
	e^{-i (\bB\bh_2)^T \br} \,d\br = \delta_{\bh_1 \bh_2},
\end{aligned}
\end{equation}
the LB energy functional $E_{LB}$ can be approximated as
\begin{equation*}
\begin{aligned}
	E_h&[\hPhi] = 
\frac{\xi^{2}}{2}\sum_{\bh_1+\bh_2=\bm{0}}[1-(\bB\bh_1)^{T}(\bB\bh_2)]^2 \hphi(\bh_1)\hphi(\bh_2)
+ \frac{\tau}{2!}\sum_{\bh_1+\bh_2=\bm{0}} \hphi(\bh_1)\hphi(\bh_2)
\\
& - \frac{\gamma}{3!} \sum_{\bh_1+\bh_2+\bh_3=\bm{0}}\hphi(\bh_1)\hphi(\bh_2)\hphi(\bh_3)
+ \frac{1}{4!} \sum_{\bh_1+\bh_2+\bh_3+\bh_4=\bm{0}}\hphi(\bh_1)\hphi(\bh_2)\hphi(\bh_3)\hphi(\bh_4)
\end{aligned}
\end{equation*}
where $\bh_j\in\bbZ^d$, Fourier coefficient $\hphi(\bh)\in X_{\bN}$, and
$\hPhi = (\hphi_1, \dots, \hphi_N)^T\in\bbC^N$. 
The convolutions in the above expression can be calculated
by Fourier pseudospectral method through the fast Fourier
transform (FFT). 
Therefore, it reduces to a finite dimensional minimization problem:
\begin{equation}\label{LB:Discrete}
	\min_{\hPhi\in \mathbb{C}^{\mathbf{N}}}
E_h[\hPhi] = G_h[\hPhi] + F_h[\hPhi]
\end{equation}
where $G_h$ and $F_h$ are the discretized interaction 
and bulk energy, respectively. The gradient of $E_h[\hPhi]$ is 
\begin{equation}
\begin{aligned}
	\nabla E_h[\hPhi] = \xi^{2} \Lambda \hPhi + \tau \hPhi 
	- \frac{\gamma}{2} \mathcal{F}_N^{-1}((\mathcal{F}_N \hPhi)^2) + 
	\frac{1}{6} \mathcal{F}_N^{-1}((\mathcal{F}_N \hPhi)^3)
\end{aligned}
\end{equation}
where $\Lambda\in\bbC^{N\times N}$ is a diagonal matrix with
entries $[1-(\bB\bh)^{T}(\bB\bh)]^2$ and
$\mathcal{F}_N\in\bbC^{N\times N}$ is the discretized Fourier
transform matrix.

\subsection{Projection method discretization}
\label{subsec:PM}

For the $d$-dimensional quasicrystals which are the space-filling structures, the
spatial integral $\frac{1}{|\Omega|}\int_{\Omega}$ in the energy
functional \eqref{eq:energy} shall be instead by the almost periodic
spacial integral $\bbint$, as defined by Eq.\,\eqref{eq:APintegral}.
We immediately have the following orthonormal property:
\begin{equation}
\begin{aligned}
	\bbint e^{i\bk \cdot \br} e^{-i\bk' \cdot \br}\,d\br =
	\delta_{\bk \bk'}, ~~~ \forall \bk, \bk' \in \bbR^d.
	\label{eq:AP:orth}
\end{aligned}
\end{equation}
For an almost periodic function, the average transformation is 
\begin{equation}
\begin{aligned}
	\hphi(\bk) = \bbint \phi(\br)e^{-i\bk\cdot \br}\,d\br, ~~
	\bk\in\bbR^d,
\end{aligned}
\end{equation}
and it is well defined \cite{katznelson2004anintroduction} .
In this paper, we carry out the above computation in a higher
dimension using the projection method which
is based on the fact that a $d$-dimensional
quasicrystal can be embedded into an $n$-dimensional periodic
structure ($n \geqslant d$)\,\cite{hiller1985crystallographic}. 
Using the projection method, the order parameter $\phi(\br)$ is
\begin{equation}
    \phi(\br) = \sum_{\bh\in\bbZ^n} \hphi(\bh)
    e^{i[(\mathcal{P}\cdot\mathbf{B}\bh)^{T}\cdot\br]},
    ~~\br\in\mathbb{R}^d,
    \label{eq:pm}
\end{equation}
where $\mathbf{B}\in\bbR^{n\times n}$ is
invertible, related to the $n$-dimensional primitive
reciprocal lattice and the projection matrix $\mathcal{P}\in\bbR^{d\times n}$ depends on the property of quasicrystals, such as
rotational symmetry\cite{hiller1985crystallographic}.
The Fourier coefficient $\hphi(\bh)$ satisfies
\begin{equation}
\begin{aligned}
	X := \left\{(\hphi(\bh))_{\bh\in\bbZ^n}:
	\hphi(\bh)\in\mathbb{C}, ~
	\sum_{\bh\in\bbZ^n}|\hphi(\bh)|<\infty \right\}.
\end{aligned}
\end{equation}
Again, in practice, we need to minimize the LP energy
functional \eqref{eq:LB} in a finite dimensional subspace. 
More precisely, let
$\bN=(N_1, N_2, \dots, N_n)\in \bbN^n$, and 
\begin{equation}
\begin{aligned}
	X_{\bN} := \{\hphi(\bh)\in X:  \hphi(\bh) = 0, 
	~\mbox{for
	all}~ |h_j|> N_j/2, ~ j=1,2,\dots,n \}.  
\end{aligned}
\end{equation}
The number of elements in the set is $N=(N_1+1)(N_2+1) \cdots (N_n+1)$. 
Together with \eqref{eq:AP:orth} and \eqref{eq:pm}, the discretized energy function \eqref{eq:LP} is
\begin{equation}
\begin{aligned}
	E_h[\hat{\Phi}] ~= ~&
	\frac{c}{2}\sum_{\bh_1 + \bh_2 = 0}
	\left[q_1^2-(\mathcal{P}\mathbf{B}\bh)^T(\mathcal{P}\mathbf{B}\bh)\right]^2
	\left[q_2^2-(\mathcal{P}\mathbf{B}\bh)^T(\mathcal{P}\mathbf{B}\bh)\right]^2
	\hphi(\bh_1)\hphi(\bh_2) 
	\\
	&+\frac{\varepsilon}{2}\sum_{\bh_1+\bh_2={\bm 0}}\hphi(\bh_1)\hphi(\bh_2) 
	-\frac{\kappa}{3}\sum_{\bh_1+\bh_2+\bh_3={\bm 0}}\hphi(\bh_1)\hphi(\bh_2)\hphi(\bh_3) 
	\\
	&
	+\frac{1}{4}\sum_{\bh_1+\bh_2+\bh_3+\bh_4={\bm
	0}}\hphi(\bh_1)\hphi(\bh_2)\hphi(\bh_3)\hphi(\bh_4),
\end{aligned}
	\label{eq:LPfinite}
\end{equation}
where $\bh_j\in\bbZ^n$, , $\hphi_j\in X_{\bN}$, $j=1,2,\dots,4$,
$\hat{\Phi}=(\hphi_1, \hphi_2, \dots, \hphi_N)\in\bbC^{N}$.
It is clear that the nonlinear
(quadratic, cubic and cross) terms in Eq.\,\eqref{eq:LPfinite} are
$n$-dimensional convolutions in the reciprocal space. A direct
evaluation of these convolution terms is extremely expensive.
Instead, these terms are simple multiplication in the
$n$-dimensional real space. Again, the efficient pseudospectral approach is
applied to calculate these convolutions in Eq.\,\eqref{eq:LPfinite}
through the $n$-dimensional FFT. 

Therefore, it leads to the following finite dimensional minimization problem:
\begin{equation}
	 \min_{\hPhi\in \mathbb{C}^{\mathbf{N}}} 
E_h[\hPhi] = G_h[\hPhi] + F_h[\hPhi],
\end{equation}
where $G_h$ and $F_h$ are the discretized interaction
and bulk energies. The gradient of $E_h[\hPhi]$ is 
\begin{equation}
\begin{aligned}
	\nabla E_h[\hPhi] = \xi^{2} \Lambda \hPhi + \tau \hPhi 
	- \frac{\gamma}{2} \mathcal{F}_N^{-1}((\mathcal{F}_N \hPhi)^2) + 
	\frac{1}{6} \mathcal{F}_N^{-1}((\mathcal{F}_N \hPhi)^3)
\end{aligned}
\end{equation}
where $\Lambda\in\bbC^{N\times N}$ is a diagonal matrix with
entries $[q_1^2-(\mathcal{P}\bB\bh)^{T}(\mathcal{P}\bB\bh)]^2
[q_2^2-(\mathcal{P}\bB\bh)^{T}(\mathcal{P}\bB\bh)]^2$.
The $\mathcal{F}_N\in\bbC^{N\times N}$ is the discretized Fourier
transform matrix and $\mathcal{F}_N^{-1}$ is the corresponding
inverse discretized Fourier transform matrix.
In the following, we will neglect the superscript of hat for simplicity.

\section{The proposed numerical approach}
\label{sec:method}
In this section, we first review the classical semi-implicit
method and accelerated proximal gradient (APG) method and then
propose the adaptive APG method with proved convergence. Finally,
the connection of generalized proximal operator with gradient flows approaches is present.

\subsection{Semi-implicit scheme}
\label{subsec:sis}

The semi-implicit scheme is a simple but useful approach for
finding the stationary state based on gradient flows. For example, 
the Allen-Cahn equation of the discretized energy functional is 
\begin{equation}
\begin{aligned}
	\Phi_t = -\nabla G_h[\Phi] - \nabla F_h[\Phi]
	\label{eq:AC}
\end{aligned}
\end{equation}
with the periodic condition and $t$ is the spurious time variable. 
Given an initial value $\Phi_0$ and the time step $\alpha$, the semi-implicit scheme is
\begin{equation}
\begin{aligned}
	\frac{1}{ \alpha }(\Phi_{k+1}-\Phi_{k}) = 
	-\nabla G_h[\Phi_{k+1}] - \nabla F_h[\Phi_k],
	\label{eq:sis}
\end{aligned}
\end{equation}
where $\Phi_k$ is the approximation of the solution at $k\alpha$, i.e., $\Phi(k \alpha )$. 
The semi-implicit scheme satisfies the following
energy dissipation property.
\begin{thm}
	\label{thm:sis}
	Let $E_h[\Phi] = F_h[\Phi] + G_h[\Phi]$. Assume that there
	exists a constant $L>0$ such that the bulk energy $F_h[\Phi]$ satisfies 
	$\max_{\Phi\in \bbC^N} \|\nabla^2 F_h[\Phi]\|_2 \leq L$, 
	and the time step length $ \alpha  \leq 1/L$, then
	the solutions of \eqref{eq:sis} satisfy 
	\begin{equation}
	\begin{aligned}
		E_h[\Phi_{k+1}] \leq E_h[\Phi_k], ~~~~ \forall k \geq 0.
	\end{aligned}
	\end{equation}
\end{thm}
\begin{proof}
	From \eqref{eq:sis}, it is easy to know
	\begin{equation*}
	\Phi_{k+1}\in\argmin_{\Phi} F_h[\Phi_k]+\langle \nabla F_h[\Phi_k],\Phi-\Phi_k\rangle + \frac{1}{2 \alpha }\|\Phi-\Phi_k\|^2+G_h[\Phi].
	\end{equation*}
	It implies
	\begin{equation*}
	\begin{aligned}
	F_h[\Phi_k]+G_h[\Phi_k]&\geq F_h[\Phi_k]+\langle\nabla
	F_h[\Phi_k],\Phi_{k+1}-\Phi_k\rangle +\frac{1}{2 \alpha }\|\Phi_{k+1}-\Phi_k\|^2+G_h[\Phi_{k+1}]\\
	&\geq F_h[\Phi_{k+1}]+G_h[\Phi_{k+1}]+\left(\frac{1}{2 \alpha }-\frac{L}{2}\right)\|\Phi_{k+1}-\Phi_k\|^2,
	\end{aligned}
	\end{equation*}
	where the last inequality is from the Taylor expansion of $F_h$ and the boundedness constraint on $\nabla^2 F_h$.
\end{proof}
Therefore, to satisfy the energy dissipation law, the time step length
$ \alpha $ depends on the Lipschitz
constant $L$. In a general PFC model, 
the universal Lipschitz constant $L$ may not exist or be very large in bounded domain which leads to a small time step and slows down the convergence speed. 
Despite its strict requirements on $ \alpha $ in theory, the semi-implicit scheme 
works well in practice which inspires us a further exploration of the semi-implicit scheme. In the following context, we will combine
modern optimization approaches and the semi-implicit scheme to obtain a
more efficient approach.

\subsection{Accelerated proximal gradient (APG) method}
\label{subsec:APG}

The classical APG method \cite{beck2009fast,tseng2008accelerated} is designed for solving the convex composite problem:
\begin{equation}\label{apg:form}
\min_{x\in\mathbb{H}}~H(x) = g(x) + f(x)
\end{equation}
where $\mathbb{H}$ is the finite dimensional Hilbert space
equipped with the inner product $\langle\cdot,\cdot\rangle$,
$g$ and $f$ are both continuously convex and $\nabla f$ has a
Lipschitz constant $L$, i.e.
\begin{equation*}
\|\nabla f(x)-\nabla f(y)\|\leq L\|x-y\|,~\forall x,y\in\mathbb{H}.
\end{equation*}
Given initializations $x_1 = x_0$ and $t_0 = 1$, the APG method consists of the following steps:
\begin{subequations}
	\begin{align}
	& t_k = (\sqrt{4(t_{k-1})^2+1}+1)/2,  \label{apg:1}\\
	&  y_k =  x_k+\frac{t_{k-1}-1}{t_k}( x_k- x_{k-1}), \label{apg:2}\\
	&  x_{k+1} = \Prox_{g}^{\alpha}( y_k - \alpha\nabla f( y_k)), \label{apg:3}
	\end{align}
\end{subequations}
where $\alpha\in(0,1/L]$ and the mapping $\Prox_g^\alpha(\cdot):\mathbb{R}^n\mapsto\mathbb{R}^n$ is defined as
\begin{equation}\label{mapping:proximal}
\Prox_g^\alpha (x) = \argmin_y~\left\{g(y) + \frac{1}{2\alpha}\|y-x\|^2\right\}.
\end{equation}
It is noted that the proximal map in \eqref{mapping:proximal} is well defined as $g$ is convex. Moreover, the step size $\alpha$ can be set adaptively as long as the following inequality holds:
\begin{equation*}
H(x_{k+1})\leq Q_{\alpha_k}(x_{k+1},x_k)\leq H(x_k)
\end{equation*}
where 
\begin{equation*}
Q_\alpha(x,y) = f(y)+\langle x-y,\nabla f(y)\rangle +\frac{1}{2\alpha}\|x-y\|^2+g(x).
\end{equation*}
The APG method has the attractive convergence property as follows.
\begin{thm}[\cite{beck2009fast}]
	Let $\{x_k\},\{y_k\}$ be the sequence generated by the \eqref{apg:1}-\eqref{apg:3} and $H^{*}$ be the optimal objective value of \eqref{apg:form} and $X_{*}$ be the set of minimizers. For any $k\geq1$, we have
	\begin{equation}\label{apg:conv}
	H(x_k)-H^{*} \leq \frac{2\|x_0-x^{*}\|^2}{\alpha(k+1)^2},\quad\forall x^{*}\in X_{*}.
	\end{equation}
\end{thm}

\subsection{Adaptive APG method}
\label{subsec:AAPG}

The discretized energy functional $E[\Phi]$ in
\eqref{LB:Discrete} can be reformulated as form \eqref{apg:form} by setting
\begin{equation}
f = F[\Phi]\quad\mbox{and} \quad g = G[\Phi].
\end{equation}
We omit the subscript $h$ for simplicity. However, there are two main obstacles when directly applying APG method for solving phase field models as $F$ is non-convex and $\nabla F$ has no universal Lipschitz constant. 
In this paper,  we propose an efficient and convergent numerical
algorithm for solving the discretized phase field model \eqref{LB:Discrete} by combining APG method with restart techniques.

The restart techniques for the APG method was proposed in \cite{o2015adaptive} which has shown significant acceleration of the APG method by imposing the decreasing property of the objective value when solving convex problems. Furthermore, another restart strategy called speed restart is developed in \cite{su2014differential} to ensure the linear convergence of the proposed restart APG method for strongly convex problems. In recent years, the restart techniques have been furtherer applied for solving non-convex composite problems in image processing \cite{bao2016image}. 
We introduce the details of the proposed algorithm in the following context.

Let $\Phi_k$ and $\Phi_{k-1}$ be the current and previous states
respectively and the extrapolation weight $w_k = (t_{k-1}-1)/t_k$. 
We can obtain a candidate state by
\begin{equation}\label{eqn:step1}
\Psi_{k+1} = \Prox_{G}^{\alpha_k}(\tilde{\Phi}_k-\alpha_k\nabla F[\tilde\Phi_k]),
\end{equation}
where 
\begin{equation}
\tilde{\Phi}_k = \Phi_k + w_k(\Phi_k-\Phi_{k-1}).
\end{equation}
It is noted that the proximal mapping in \eqref{eqn:step1} is well defined as $G$ is convex.
Different from the APG method, the restart technique is to determine whether we accept the result $\Psi_{k+1}$ as the new estimate $\Phi_{k+1}$. Inspired by the function value restart condition in \cite{o2015adaptive}, we choose $\Phi_{k+1} = \Psi_{k+1}$ whenever the following condition holds:
\begin{equation}\label{res:fun}
E[\Phi_k]-E[\Psi_{k+1}]\geq \delta\|\Phi_k-\Psi_{k+1}\|^2
\end{equation}
for some $\delta>0$. If the condition \eqref{res:fun} does not
hold, we restart the APG by setting $w_k=0$. In this case, we have 
\begin{equation}\label{eqn:semi_implicit}
\Phi_{k+1} = \Prox_{G}^{\alpha_k}(\Phi_k-\alpha_k\nabla F(\Phi_k)).
\end{equation}
In fact, the scheme \eqref{eqn:semi_implicit} provides an
adaptive time step semi-implicit approach when $\alpha_k$ varies.
From the continuity of $F$, $G$ in
\eqref{LB:Discrete} and the coercive property of $F$, i.e.\
\begin{equation}
F(\Phi)\rightarrow +\infty,\quad\Phi\rightarrow+\infty,
\end{equation}
the
sub-level set $\{E\leq
E[\Phi_0]\}=\{\Phi\in\mathbb{H}~|~E[\Phi]\leq E[\Phi_0]\}$ is
compact for any $\Phi_0$. Let $\mathcal{M}$ be the closed ball
that contains $[E\leq E[\Phi_0]]$. From the smoothness of $F$,
$\nabla F$ is Lipschitz continuous in $\mathcal{M}$. Denote
$L_{\mathcal{M}}$ to be the Lipschitz constant of $\nabla F$ in
the set $\mathcal{M}$, i.e.\
\begin{equation*}\label{Lip:M}
\|\nabla F[\Phi]-\nabla F[\Psi]\|\leq L_{\mathcal{M}}\|\Phi-\Psi\|,~\forall \Phi,\Psi\in\mathcal{M}.
\end{equation*}
Thus, we obtain the the next proposition  that shows
$\Phi_k\in\mathcal{M}$ satisfying the sufficient decrease condition for all $k$.
\begin{proposition}\label{prop:suffDec}
	Given an initial point $\Phi_0$ and the iterates $\Phi_{k+1}
	= \Prox_G^{\alpha_k}(\Phi_k-\alpha_k\nabla F(\Phi_k))$ with
	$\alpha_k\in(0,1/L_{\mathcal{M}})$ for $k=0,1,\ldots$. If
	$\{\Phi_k\}_{k=1}^{\infty}\subset \mathcal{M}$, then 
	\begin{equation}\label{suff_decrease}
	E[\Phi_k] - E[\Phi_{k+1}]\geq (1/2\alpha_k-L_{\mathcal{M}}/2)\|\Phi_{k+1}-\Phi_{k}\|^2,~\forall k = 1,2,\ldots.
	\end{equation}
\end{proposition}
The proof can be easily obtained from Theorem \ref{thm:sis}. Let $\Phi_{k+1}$ is from \eqref{eqn:semi_implicit} and $\eta\leq\frac{1}{2\alpha_k}-\frac{L_\mathcal{M}}{2}$, the following sufficient condition 
\begin{equation}\label{criterion:line_search}
E[\Phi_k] - E[\Phi_{k+1}]\geq \eta\|\Phi_{k+1}-\Phi_k\|^2
\end{equation}
holds and thus the \eqref{eqn:semi_implicit} is a safe-guard step which ensures energy dissipation. On the other hand, by Proposition \ref{prop:suffDec}, $\alpha_k$ should be less than $1/L_{\mathcal{M}}$ which might be very small. Thus, it only allows a small step size which may significantly slow down the convergence, and be always too conservative. By line search technique, we can adaptively estimate the step size $\alpha_k$ which will be introduced as follows.

{\bf \noindent Estimation of step size $\alpha_k$.}
Define: $s_{k-1}:=\Phi_k-\Phi_{k-1}$, and $g_{k-1}:=\nabla
F[\Phi_k]-\nabla F[\Phi_{k-1}]$. We initialize the search step by
the Barzilai-Borwein (BB) method \cite{barzilai1988two}, i.e.\
\begin{equation}\label{BB_init}
\beta_0 = \frac{\langle s_{k-1}, s_{k-1}\rangle}{\langle s_{k-1}, g_{k-1}\rangle}\quad\mbox{or}\quad
\beta_0 = \frac{\langle s_{k-1}, g_{k-1}\rangle}{\langle g_{k-1}, g_{k-1}\rangle}.
\end{equation}
Together with the standard backtracking, we adopt the step size $\alpha_k$ whenever \eqref{criterion:line_search} holds.
The detailed algorithm of estimation the step size $\alpha_k$ is given in Algorithm \ref{alg:eststep}.
\begin{algorithm}[!pbht]
	\caption{Estimation of $\alpha_k$ at $\Psi_k$}
	\label{alg:eststep}
	\begin{algorithmic}[1]
		\State {\bf Inputs:} $\Phi_k$, $\Psi_k$, $\nabla
		F[\Psi_k]$, $\nabla F[\Phi_k]$, $\rho\in(0,1)$ and $\eta>0$ 
		\State {\bf Output:} step size $\alpha_k$
		\State Set $ s_k =  \Psi_k- \Phi_k$ and $d_k = \nabla
		F[\Psi_k]-\nabla F[\Phi_k]$.
		\State Initialize $\beta$ by the Barzilai-Borwein method
		via Eqn.\eqref{BB_init}
		\For{$j=1,2\ldots$}
		\State Calculate $ \Psi_{k+1} = \Prox_{G}^{\beta}( \Psi_k - \beta\nabla F[\Psi_k])$
		\State Step size length $\beta$ is obtained by the linear search technique
		\If{$E[\Psi_k]-E[\Psi_{k+1}]\geq\eta\|\Psi_k-\Psi_{k+1}\|^2$}
		\State $\alpha_k=\beta$ and {\bf break}
		\Else
		\State $\beta = \rho\beta$
		\EndIf 
		\EndFor
	\end{algorithmic}
\end{algorithm}
and the proposed adaptive APG algorithm is present in Algorithm \ref{alg:adaAPG}.
\begin{algorithm}[!pbht]
	\caption{Adaptive APG algorithm for PFC model}
	\label{alg:adaAPG}
	\begin{algorithmic}[1]
		\State Initialize $\Phi_1=\Phi_0$, $w_0\in [0,1]$, $ N_{max}\in\mathbb{N}$, $k_{ada}=0$, $\eta\geq\delta>0$.
		\For{$k=1,2,3,\ldots,$}
		\State Update $w_k\in [0,1]$
		\State Update $\Psi_k = (1+w_k)\Phi_k - w_k\Phi_{k-1}$
		\State Estimate the step size $\alpha_k$ at $\Psi_k$ via Algorithm~\ref{alg:eststep}
		\State Calculate $\Psi_{k+1} =
		\Prox_{G}^{\alpha_k}(\Psi_k-\alpha_k\nabla F[\Psi_k])$.
		\If{$E[\Phi_k]-E[\Psi_{k+1}]\geq \delta
		\|\Phi_k-\Psi_{k+1}\|^2$ holds and $k-k_{ada}\leq N_{max}$}
		\State Set $\Phi_{k+1} = \Psi_{k+1}$.
		\Else
		\State Reset $w_k=0$ and $k_{ada} = k$.
		\EndIf
		\EndFor
	\end{algorithmic}
\end{algorithm}

\subsubsection{Convergence analysis}
\label{subsec:convergence}

In this section, we show that our proposed method converges to a steady state of the original energy function. Firstly, we present a useful lemma for our analysis.
\begin{lemma}[Uniformized Kurdyka-Lojasiewicz property
	\cite{bolte2014proximal}.] Let $\Omega$ be a compact set and
	$E$ defined in \eqref{LB:Discrete} be bounded below. Assume that $E$ is constant on $\Omega$. Then, there exist $\epsilon>0$, $\eta>0$, and $\psi\in\Psi_\eta$ such that for all $\bar u\in\Omega$ and all $u\in\Gamma_\eta(\bar u,\epsilon)$, 
	one has,
	\begin{equation}\label{UKL}
	\psi^{'}(E(u)-E(\bar u))\|\nabla E(u)\|\geq 1,
	\end{equation}
	where $\Psi_\eta =\{\psi\in C[0,\eta)\cap C^1(0,\eta) \text{ and } \psi \text{ is concave}, \psi(0)=0, \psi^{'}>0 \text{ on } (0,\eta)\}$ and $\Gamma_\eta(x,\epsilon) = \{y|\|x-y\|\leq \epsilon, E(x)<E(y)<E(x)+\eta\}$.
	\label{lemma:ukl}
\end{lemma}
\begin{proof}
	The proof is based on the facts that $F$ and $G$ satisfy the so called Kurdyka-Lojasiewicz property on $\Omega$ \cite{bolte2014proximal}.
\end{proof}

\begin{thm}\label{thm:convergence}
	Let $E$ defined in \eqref{LB:Discrete} be bounded below and
	$\{\Phi_k\}$ be the sequence generated by
	Algorithm~\ref{alg:adaAPG}. If $\Phi_k\in \mathcal{M}$ and
	$\liminf_k \alpha_k=\bar\alpha>0$, then $\{\Phi_k\}$ has the 
	global convergence property, i.e.\ there exists a point
	$\Phi^*$ such that $\nabla E[\Phi^*]=\vzero$ and
	$\lim\limits_{k\rightarrow+\infty} \Phi_k=\Phi^*$.
\end{thm}
\begin{proof}
	Define 
	\begin{equation}\label{iter	x_update}
	P_{k+1} = \Prox_G^{\alpha_k}(\Phi_k-\alpha_k\nabla F[\Phi_k])
	\end{equation}
	and  two sets $\Omega_2 = \{k\,|\,t_k = 1\}$ and $\Omega_1=\mathbb N\backslash\Omega_2$. It is noted that for any $k\in\Omega_2$, we have $\Phi_{k+1}=P_{k+1}$.
	Let $w_k=(t_{k}-1)/t_{k+1}$, then there exists some $\bar w=
	(t_{N_{max}}-1)/(t_{N_{max}}+1) \in[0,1)$ such that $w_k\leq \bar w$ for all $k$ as $t_k$ is increasing and $t_k$ is reset to $1$ at most every $N_{max}$ iteration. We show the following properties of the sequence $\{x_k\}$ generated by Algorithm~\ref{alg:adaAPG}.\\
	{\bf Sufficient decrease property.} If $k\in\Omega_2$, we have 
	\begin{equation}
	E[\Phi_k] - E[\Phi_{k+1}] \geq\max( 1/2\alpha_k-L_{\mathcal{M}}/2,\eta)\|\Phi_k-\Phi_{k+1}\|^2
	\end{equation}
	from Proposition \ref{prop:suffDec} and the line search criterion \eqref{criterion:line_search}.
	Together with the condition \eqref{res:fun},  the following \emph{sufficient decrease property} holds	
	\begin{equation}\label{suffDec}
	E[\Phi_k] - E[\Phi_{k+1}] \geq \rho_1\|\Phi_k-\Phi_{k+1}\|^2,~\forall k,
	\end{equation}
	where $\rho_1 = \min\{\eta,\delta\}>0$. 
	Since $\inf E>-\infty$, there exists $E^{*}$ such that
	$E[\Phi_k]\geq E^{*}$ and $E[\Phi_k]\rightarrow E^{*}$ as $k\rightarrow+\infty$. This implies
	\begin{equation}\label{incre_x}
	\rho_1\sum_{k=0}^\infty \|\Phi_{k+1}-\Phi_k\|^2\leq
	E[\Phi_0] - E^{*}<+\infty \text{ and } \lim\limits_{k\rightarrow+\infty}\|\Phi_{k+1}-\Phi_k\| = 0.
	\end{equation}
	{\bf Bounded the gradient.} If $k\in\Omega_1$, by the optimality condition of \eqref{eqn:step1}, we have 
	\begin{equation*}
	-\nabla F[\tilde{\Phi}_k] + \frac{1}{\alpha_k}
	(\tilde{\Phi}_k-\Phi_{k+1})=\nabla G[\Phi_{k+1}].
	\end{equation*}
	Thus, $ \nabla F[\Phi_{k+1}]-\nabla F[\tilde{\Phi}_k] +
	\dfrac{1}{\alpha_k} (\tilde{\Phi}_k-\Phi_{k+1})=\nabla E[\Phi_{k+1}]$ and 
	\begin{equation}\label{subBound_1}
	\begin{aligned}
	\|\nabla E[\Phi_{k+1}]\|&\leq (L_{\mathcal{M}}+1/\bar{\alpha})\|\Phi_{k+1}-\tilde{\Phi}_k\|\\
	&\leq(L_{\mathcal{M}}+1/\bar{\alpha})(\|\Phi_{k+1}-\Phi_k\|+w_k\|\Phi_k-\Phi_{k-1}\|),
	\end{aligned}
	\end{equation}
	as $\tilde{\Phi}_k\in\mathcal{M}$ where $\mathcal{M}$ is a bounded set and $L_{\mathcal{M}}$ is the Lipschitz constant of $\nabla F$ in $\mathcal{M}$.
	If $k\in\Omega_2$, by the optimality condition of \eqref{iter	x_update}, we have
	\[
	-\nabla F[\Phi_k] + \frac{1}{\alpha_k}(\Phi_k-\Phi_{k+1})= \nabla G[\Phi_{k+1}].
	\]
	Then, $ \nabla F[\Phi_{k+1}]-\nabla F[\Phi_k] +
	\dfrac{1}{\alpha_k}(\Phi_k-\Phi_{k+1})= \nabla E[\Phi_{k+1}]$, then
	\begin{equation}\label{subBound_2}
	\|\nabla E[\Phi_{k+1}]\|\leq (L_{\mathcal{M}}+1/\bar{\alpha})\|\Phi_{k+1}-\Phi_k\|.
	\end{equation}
	Combining \eqref{subBound_1} with \eqref{subBound_2}, it follows that 
	\begin{equation}\label{subBound}
	\|\nabla E[\Phi_{k+1}]\|\leq \rho_2 (\|\Phi_{k+1}-\Phi_k\|+w_k\|\Phi_k-\Phi_{k-1}\|)\leq\rho_2 (\|\Phi_{k+1}-\Phi_k\|+\bar w\|\Phi_k-\Phi_{k-1}\|),
	\end{equation}
	where
$\rho_2=L_{\mathcal{M}}+1/\bar{\alpha}>0$.\\
	{\bf Subsequence convergence.} Since
	$\{\Phi_k\}\subset\mathcal{M}$ which is compact, there exists a subsequence $\{\Phi_{k_j}\}$ and $\Phi^{*}\in\mathcal{M}$ such that 
	\begin{equation}
	\lim_{j\rightarrow+\infty}\Phi_{k_j} = \Phi^{*},\quad
	\lim_{j\rightarrow+\infty} E[\Phi_{k_j}] = E[\Phi^{*}]
	\quad\mbox{and}\quad \lim_{j\rightarrow+\infty} \nabla
	E[\Phi_{k_j}] = \nabla E[\Phi^{*}],
	\end{equation}
	where the last two equalities are from the continuity of $E$. Moreover, \eqref{incre_x} implies
	\begin{equation}
	\lim_{j\rightarrow+\infty} \|\Phi_{k_j}-\Phi_{k_j-1}\| = 0
	\quad\mbox{and}\quad \lim_{j\rightarrow+\infty}
	\|\Phi_{k_j-1}-\Phi_{k_j-2}\| = 0.
	\end{equation}
	Then, we know $\nabla E[\Phi^{*}]=\vzero$ from \eqref{subBound}.\\
	{\bf Finite length property.} Let $\omega(\Phi_0)$ be the set
	of limiting points of the sequence $\{\Phi_k\}$ starting from
	$\Phi_0$. By the boundedness of $\{\Phi_k\}$ and the fact
	$\omega(\Phi_0)=\cap_{q\in\mathbb{N}}\overline{\cup_{k\geq
	q}\{\Phi_k\}}$, it follows that $\omega(\Phi_0)$ is a
	non-empty and compact set. Moreover, from \eqref{suffDec}, we
	know $E[\Phi]$ is constant on $\omega(\Phi_0)$, denoted by
	$E^*$. If there exists some $k_0$ such that
	$E[\Phi_{k_0}]=E^*$, then we have $E[\Phi_k]=E^*$ for all
	$k\geq k_0$ which is from \eqref{suffDec}. In the following
	proof, we assume that $E[\Phi_k]>E^*$ for all $k$. Therefore,
	$\forall \epsilon,\eta>0$, there exists some $\ell>0$ such
	that for all $k>\ell$, we have
	$\mathrm{dist}(\omega(\Phi_0),\Phi_k)\leq \epsilon \text{ and
	} E^*<E[\Phi_k]<E^*+\eta$, i.e. 
	\begin{equation}\label{APP:EQ}
	\Phi\in\Gamma_{\eta}(\Phi^*,\epsilon) \quad \text{for all }\quad\Phi^*\in w(\Phi_0).
	\end{equation}
	Applying lemma \ref{lemma:ukl}, for all $k>\ell$ we have
	\begin{equation*}
	\psi^{'}(E[\Phi_k]-E^*)\|\nabla E[\Phi_k]\|\geq 1.
	\end{equation*}
	Form \eqref{subBound}, it implies
	\begin{equation}\label{KL1}
	\psi^{'}(E[\Phi_k]-E^*)\geq
	\frac{1}{\rho_2(\|\Phi_k-\Phi_{k-1}\|+w_{k-1}\|\Phi_{k-1}-\Phi_{k-2}\|)}.
	\end{equation}
	By the convexity of $\psi$, we have
	\begin{equation}\label{concave}
	\psi(E[\Phi_k]-E^*) - \psi(E[\Phi_{k+1}] -E^*)
	\geq \psi^{'}(E[\Phi_k]-E^*)(E[\Phi_k]-E[\Phi_{k+1}]).
	\end{equation}
	Define $\Delta_{p,q}=\psi(E[\Phi_p]-E^*) - \psi(E[\Phi_q] -E^*)$ and $C=\rho_2/\rho_1>0$. Together with \eqref{KL1}, \eqref{concave} and \eqref{suffDec}, we have for all $k>\ell$
	\begin{equation}
	\Delta_{k,k+1}\geq \frac{\|\Phi_{k+1}-\Phi_k\|_2^2}{C(\|\Phi_k-\Phi_{k-1}\|+w_{k-1}\|\Phi_{k-1}-\Phi_{k-2}\|)},
	\end{equation}
	and therefore,
	\begin{equation}\label{GeoIneq}
	2\|\Phi_{k+1}-\Phi_{k}\|\leq \|\Phi_k-\Phi_{k-1}\|+w_{k-1}\|\Phi_{k-1}-\Phi_{k-2}\| + C\Delta_{k,k+1},
	\end{equation}
	which is from the fact that geometric inequality. For any $k>\ell$, summing up \eqref{GeoIneq} for $i=\ell+1,\ldots,k$, it implies
	\begin{equation*}
	\begin{split}
	& 2\sum_{i=\ell+1}^k \|\Phi_{i+1}-\Phi_{i}\|\leq \sum_{i=\ell+1}^k (\|\Phi_i-\Phi_{i-1}\|+w_{i-1}\|\Phi_{i-1}-\Phi_{i-2}\|)+C\sum_{i=\ell+1}^k\Delta_{i,i+1}\\
	\leq & \sum_{i=\ell+1}^k (1+w_i)\|\Phi_{i+1}-\Phi_i\| + (1+w_\ell)\|\Phi_\ell-\Phi_{\ell-1}\|+w_{\ell-1}\|\Phi_{\ell-1}-\Phi_{\ell-2}\|+C\Delta_{\ell+1,k+1},
	\end{split}
	\end{equation*}
	where the last inequality is from the fact that $\Delta_{p,q}+\Delta_{q,r}=\Delta_{p,r}$ for all $p,q,r\in\mathbb{N}$. Since $\psi\geq 0$, for any $k>\ell$ and $w_k\leq\bar w$, we have
	\begin{equation}\label{bound	GlobalConv}
	\begin{split}
	\sum_{i=\ell+1}^k(1-\bar w)\|& \Phi_{i+1} -\Phi_i\| \leq\sum_{i=\ell+1}^k(1-w_i)\|\Phi_{i+1}-\Phi_i\|\\
	& \leq
	(1+w_\ell)\|\Phi_\ell-\Phi_{\ell-1}\|+w_{\ell-1}\|\Phi_{\ell-1}-\Phi_{\ell-2}\|+C\psi(E[\Phi_\ell]-E^*).
	\end{split}
	\end{equation}
	This easily implies that
	$\sum_{k=1}^{\infty}\|\Phi_{k+1}-\Phi_k\|<\infty$ and
	$\lim\limits_{k\rightarrow+\infty}\Phi_k=\Phi^*$ where
	$\nabla E[\Phi^{*}]=\vzero$.
\end{proof}

\section{Connection with gradient flows}
\label{sec:connection}
Let $\mathcal{L}$ be a non-positive symmetric operator, the gradient flow of energy $E$ can be formulated as
\begin{equation}\label{op:gradFlow}
\frac{\partial \phi}{\partial t} = \mathcal{L}\frac{\delta E}{\delta \phi}. 
\end{equation}
Two classical gradient flow approaches for solving the PFC model are
\begin{subequations}
\begin{align}
\text{(Allen-Cahn)}\quad&\frac{\partial \phi}{\partial t}=-\frac{\delta E}{\delta \phi}, \label{ac:model}\\
\text{(Cahn-Hilliard)}\quad&\frac{\partial \phi}{\partial t}=\nabla\cdot\left(M_{\phi}\nabla\frac{\delta E}{\delta\phi}\right),\label{ch:model}
\end{align}
\end{subequations}
with appropriate boundary conditions where $M_\phi$ is a
non-positive symmetric operator dependent on $\phi$.
Again, splitting $E[\phi]$ into $E[\phi]=F[\phi]+G[\phi]$, for a given
spacial discretization, the discretized energy can be formulated as 
\begin{equation}
E_h[\Phi] = F_h[\Phi]+G_h[\Phi].
\end{equation}
Typical first-order numerical approaches for solving \eqref{op:gradFlow} include explicit, semi-implicit and implicit schemes, i.e.\
\begin{equation}\label{schemes_numeric}
\frac{\Phi_{k+1}-\Phi_k}{ \alpha } = \mathcal{L}_h\begin{cases}
\nabla F_h[\Phi_k] + \nabla G_h[\Phi_k], & \mbox{(Explicit)},\\
\nabla F_h[\Phi_k] + \nabla G_h[\Phi_{k+1}],& \mbox{(Semi-implicit)},\\
\nabla F_h[\Phi_{k+1}] + \nabla G_h[\Phi_{k+1}],& \mbox{(Implicit)},
\end{cases}
\end{equation}
where $\mathcal{L}_h$ denotes the discretization of $\mathcal{L}$. 
To build up the connection with \eqref{schemes_numeric}, we define the generalized proximal operator.
\begin{defy}[Generalized proximal operator] Let $G$ be a proper, lower semi-continuous function and $\mathcal{S}$ be a positive symmetric operator. The generalized proximal operator with respect to $\mathcal{S}$ is
	\begin{equation}\label{prox:general}
	\GProx_{G,\mathcal{S}} (y) = \argmin_x \left\{G(x)+\frac{1}{2}\|x-y\|_{\mathcal{S}}^2\right\},
	\end{equation}
	where $\|x\|_{\mathcal{S}}^2=\langle x,\mathcal{S}x\rangle$.
\end{defy}
It is noted $\GProx_{G,\mathcal{S}}$ is non-empty and compact, see \cite{bolte2014proximal}. The connection between the generailized proximal operator and scheme \eqref{schemes_numeric} arrives the following proposition:
\begin{proposition}\label{prop:connection}
	If $\mathcal{L}_h$ is invertible. The schemes in \eqref{schemes_numeric} are equivalent to
	\begin{equation}\label{connection}
	\Phi_{k+1}=\begin{cases}
	\GProx_{\vzero,\mathcal{I}}\left(\Phi_k+ \alpha \mathcal{L}_h\left(\nabla F_h[\Phi_k] + \nabla
	G_h[\Phi_k]\right)\right), & \mbox{\rm(Explicit scheme)},    \\ 
	\GProx_{ \alpha  G_h,-\mathcal{L}_h^{-1}}\left(\Phi_k + \alpha \mathcal{L}_h\nabla F_h[\Phi_k]\right), &
	\mbox{\rm(Semi-implicit scheme)},
	\\
	\GProx_{\alpha (F_h+G_h),-\mathcal{L}_h^{-1}}\left(\Phi_k\right), &
	\mbox{\rm(Implicit scheme)},
	\end{cases}
	\end{equation}
	where $\mathcal{I}$ is the identity operator.
\end{proposition}
\begin{proof}
	As the proof of three schemes are similar, we only prove the semi-implicit case.
	It is noted that $-\mathcal{L}_h$ is positive symmetric as $\mathcal{L}_h$ is non-negative and invertible. Since $\Phi_{k+1}=	\GProx_{\alpha G_h,-\mathcal{L}_h^{-1}}\left(\Phi_k+ \alpha \mathcal{L}_h\nabla
	F_h[\Phi_k]\right)=\argmin\limits_\Phi\big\{\alpha G_h[\Phi]+\frac{1}{2}\|\Phi-\Phi_k- \alpha  \mathcal{L}_h
	\nabla F_h[\Phi_k] \|_{-\mathcal{L}_h^{-1}}^2\big\}$, we have
	\begin{equation}
	0= \alpha \nabla G_h[\Phi_k] -
	\mathcal{L}_h^{-1}(\Phi_{k+1}-\Phi_k-\alpha \mathcal{L}_h\nabla F_h[\Phi_k])
	\end{equation}
	from the first order optimality condition which implies semi-implicit numerical scheme.
\end{proof}
\begin{remark}
	It is noted that $\mathcal{L}_h=-\mathcal{I}$ in Allen-Cahn equation and $\mathcal{L}_h = \Delta$ in Cahn-Hilliard when $M_\phi=1$. Based on our analysis, it is suggested that  $\mathcal{L}_h=\Delta-\tau\mathcal{I}$ for some $\tau>0$ when applying the Cahn-Hilliard equation. The implicit scheme for $\mathcal{L}_h=\Delta-\tau I$ is the gradient step of the viscosity solution for certain Hamilton-Jacobi equation as pointed out in \cite{osher2018laplacian}. Moreover, If $\mathcal{L}_h=(\Delta-\tau I)^{-1}$ for some $\tau>0$ in \eqref{connection}, the explicit scheme is the (generalized) Laplacian smoothing introduced \cite{osher2018laplacian}.
\end{remark}
\begin{remark}
	The APG method can be formulated as 
	\begin{equation*}
	\Phi_{k+1} = \GProx_{\alpha G_h,\mathcal{I}}\left(\tilde\Phi_k- \alpha \nabla
	F_h[\tilde\Phi_k]\right),~\tilde\Phi_k = \Phi_k+w_k(\Phi_k-\Phi_{k-1}),
	\end{equation*}
	for some $w_k\in(0,1)$. When the objective function is
	convex, the extrapolation step has been proved to accelerate
	convergence. Meanwhile, from the perspective of interpolation
	methods, $\tilde\Phi_k$ can also be thought as an
	approximation of the implicit step. It is the Lagrange
	interpolation when $w_k=1$.
\end{remark}
\begin{remark}
	The energy dissipation is related to the objective
	value decreasing property of the generalized proximal
	operators in \eqref{connection}; the adaptive time stepping
	corresponds to the adaptive step sizes $\alpha_k$ which can
	be efficiently implemented by the line search as shown in
	Algorithm \ref{alg:eststep}.
\end{remark}
As the semi-implicit approach is not unconditional energy
dissipation, stabilized methods have been proposed
\cite{xu2006stability}.
In concrete, the stabilized semi-implicit scheme contains 
\begin{equation}\label{stat_sis}
\frac{\Phi_{k+1}-\Phi_k}{ \alpha } = \mathcal{L}_h(\nabla
G_h[\Phi_{k+1}]+\nabla F_h[\Phi_k]+\sigma(\Phi_{k+1}-\Phi_k)).
\end{equation}
for some $\sigma>0$. Suppose $\mathcal{L}_h$ is invertible and  $\mathcal{S}=-(\mathcal{I}-\sigma \alpha \mathcal{L}_h)^{-1}\mathcal{L}_h$ is positive symmetric, the above scheme is equivalent to 
\begin{equation}
\begin{aligned}
& (\mathcal{I}-\sigma  \alpha \mathcal{L}_h)(\Phi_{k+1}-\Phi_k) =
\mathcal{L}_h \alpha ( \nabla G_h[\Phi_{k+1}]+\nabla F_h[\Phi_k])\\
\Longleftrightarrow \quad & \Phi_{k+1} = \GProx_{ \alpha  G_h,
\mathcal{S}^{-1}}(\Phi_k- \alpha \mathcal S\nabla F_h[\Phi_k]).
\end{aligned}
\end{equation}
In Allen-Cahn equation, $\mathcal{L}_h=-\mathcal{I}$, all the
required conditions are automatically satisfied. However, in
Cahn-Hilliard equation, the corresponding conditions require
further exploration.
In general case, discovering the deep connections between the
gradient flow and the proximal operators may provide new insights
for both fields and we will explore it in future.

\section{Numerical results and discussions}
\label{sec:result}

In this section, we present several numerical examples to
illustrate the efficiency and accuracy of our method by comparing
with the semi-implicit scheme (SIS). 
All experiments were performed on a workstation with a 3.20 GHz
CPU (i7-8700, 12 processors). All code were written by MATLAB
language without parallel implementation.
In our experiments, the Algorithm \ref{alg:adaAPG} is employed to
calculate the stationary states of finite dimensional PFC models,
including the LB energy functional \eqref{eq:LB} with the Fourier
pseudospectral discretization $E_h$ (see
Eq.\,\eqref{LB:Discrete}) for periodic crystals and the
LP energy functional with the projection
method discretization for quasicrystals.  
Let $\Phi_s$ be the ``exact'' stationary state obtained numerically
with a very fine mesh and $E_s = E_h[\Phi_s]$ be its energy.
Correspondingly, let $\Phi_{s,h}$ be the numerical stationary
state obtained with the mesh size $h$ and $E_h[\Phi_{s,h}]$ be
its energy.

\subsection{Periodic crystals}

For the LB model, we use three dimensional periodic crystals of
the double gyroid and the sigma phase,
recently discovered both in polymer experiments and in theoretical
computations\,\cite{zhang2008efficient, lee2010discovery}, to
demonstrate the performance of our approach. 

\subsubsection{Double gyroid}
\label{subsec:dg}

The double gyroid phase is a continuous network periodic phase. 
Its initial value is 
\begin{equation}
\begin{aligned}
	\phi(\br) = \sum_{\bh\in\Lambda_0^{DG}} \hphi(\bh) e^{i (\bB\bh)^T
	\cdot \br },
	\label{eq:LB:initial}
\end{aligned}
\end{equation}
where initial lattice points set $\Lambda_0^{DG}\subset\bbZ^3$ only on
which the Fourier coefficients located are nonzero.
The corresponding $\Lambda_0^{DG}$ of the double gyroid phase
is given in the Table \ref{tab:DG:initial}. For more details, please refer to \cite{jiang2013discovery}.
\begin{table}[htbp]
  \centering
  \caption{The initial lattice points set $\Lambda_0$ of the double
  gyroid phase. $^o$ denotes the sign of Fourier coefficients is opposite.
  }
  \label{tab:DG:initial}
  \begin{tabular}{|c|c|}
 \hline
 $\Lambda_0^{DG}$  &
\makecell{
				$(-2,1,1)$,
				$(2,1,1)^o$, 
				$(2,1,-1)^o$, 
				$(2,-1,1)$, 
				$(1,-2,1)$, 
				$(1,2,-1)$,
			\\  
		  		$(1,2,1)^o$, 
				$(-1,2,1)^o$,
                $(1,1,-2)$, 
				$(1,-1,2)^o$, 
				$(-1,1,2)$,
                $(1,1,2)^o$
				}
               \\\hline
   \end{tabular}
\end{table}
The double gyroid structure belongs to the cubic crystal system,
therefore, the $3$-order invertible matrix can be chosen as 
$\bB= (1/\sqrt{6})\bI_3$. Correspondingly, the computational
domain in physical space is $[0,2\sqrt{6}\pi)^3$.
The parameters in LB model \eqref{eq:LB} are set as 
$\xi = 0.1, \tau = -2.0, \gamma = 2.0$.

The exact solution is obtained numerically by using $256\times 256
\times 256$ spatial discretization points, and its exact energy 
with such model parameters is $E_s = -12.9429155189828$. 
Table \ref{tab:DG:accuracy} presents the
numerical error for the double gyroid phase. From Table
\ref{tab:DG:accuracy}, it is observed that the Fourier pseudospectral method
is spectral accuracy. 
Figure \ref{fig:DG} shows the morphology of stationary double
gyroid phase.

\begin{table}[htbp]
	\centering
	\caption{Accuracy of the Fourier pseudospectral discretization for
	the double gyroid phase in the LB model simulations.  
	The solution with $ 256 \times 256 \times 256 $ is used as reference solution.
	}
	\label{tab:DG:accuracy}
	\begin{tabular}{|c|c|c|c|c|}
	\hline
	 & DOF & $ 32^3 $ & $ 64^3 $ & $ 128^3 $ 
	\\ \hline 
	Double gyroid   &
	\makecell{
	$\|\Phi_s - \Phi_h\|_2$
	\\
	$|E_s - E_h(\Phi_{s,h})|$
	}
	& 
	\makecell{ 6.2770e-05 \\ 4.9949e-02 
	}
	& 
	\makecell{ 7.7191e-08 \\ 2.3984e-06 
	}
	& 
	\makecell{ 7.0668e-12 \\ 1.0658e-14 
	}
	\\ \hline 
   \end{tabular}
\end{table}
\begin{figure}[htbp]
	\centerline{\includegraphics[scale=0.4]{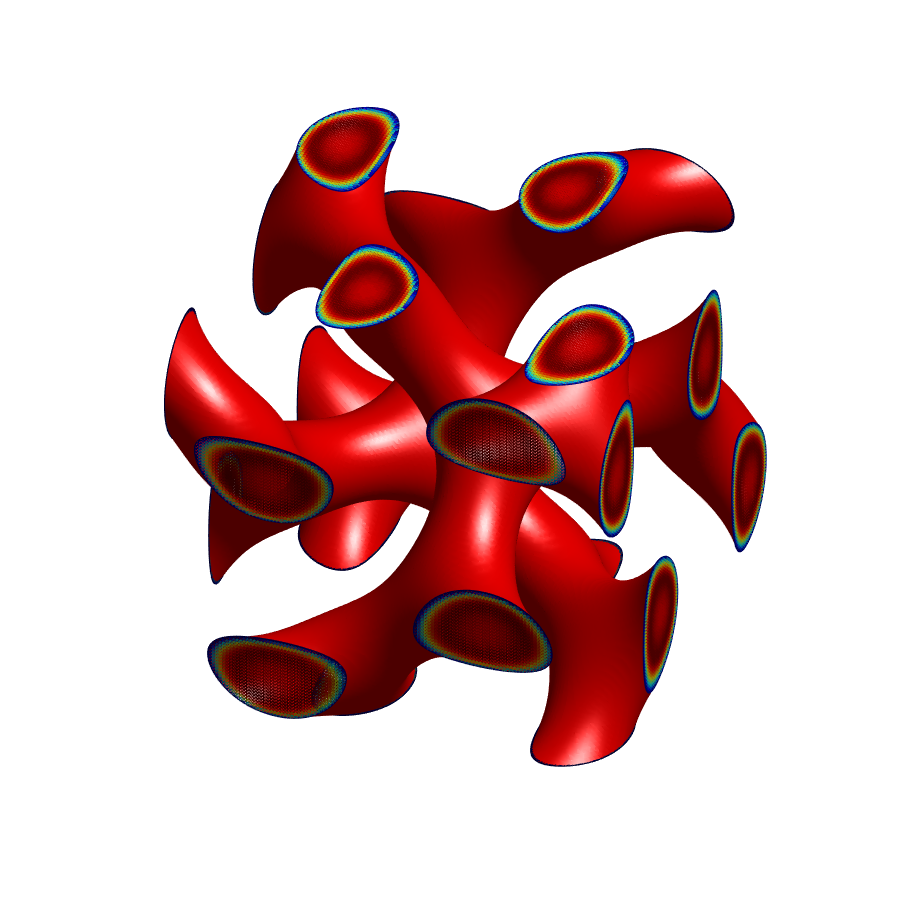}}
	\caption{\label{fig:DG}
{\small	The stationary double gyroid phase in LB model
	with $\xi = 0.1, \tau = -2.0, \gamma = 2.0$. }
	}
\end{figure}

In order to demonstrate the performance of our proposed method, a
convergent comparison between the adaptive APG method and the 
SIS for the energy difference is shown in Figure
\ref{fig:DG:comparison}, using $128\times 128 \times 128$ spacial
discretization points.  
In the SIS, the time step $ \alpha $ is fixed, while in adaptive APG
approach, $ \alpha $ can be obtained adaptively by the linear search
technique, as given in Figure \ref{fig:DG:step}.
In comparison, the fixed time step of the SIS is chosen 
as $0.2$ to guarantee the best performance on the premise
of energy dissipation. It is shown that the adaptive APG algorithm converges
faster than the SIS.
\begin{figure}[htbp]
	\centering
	\includegraphics[scale=0.4]{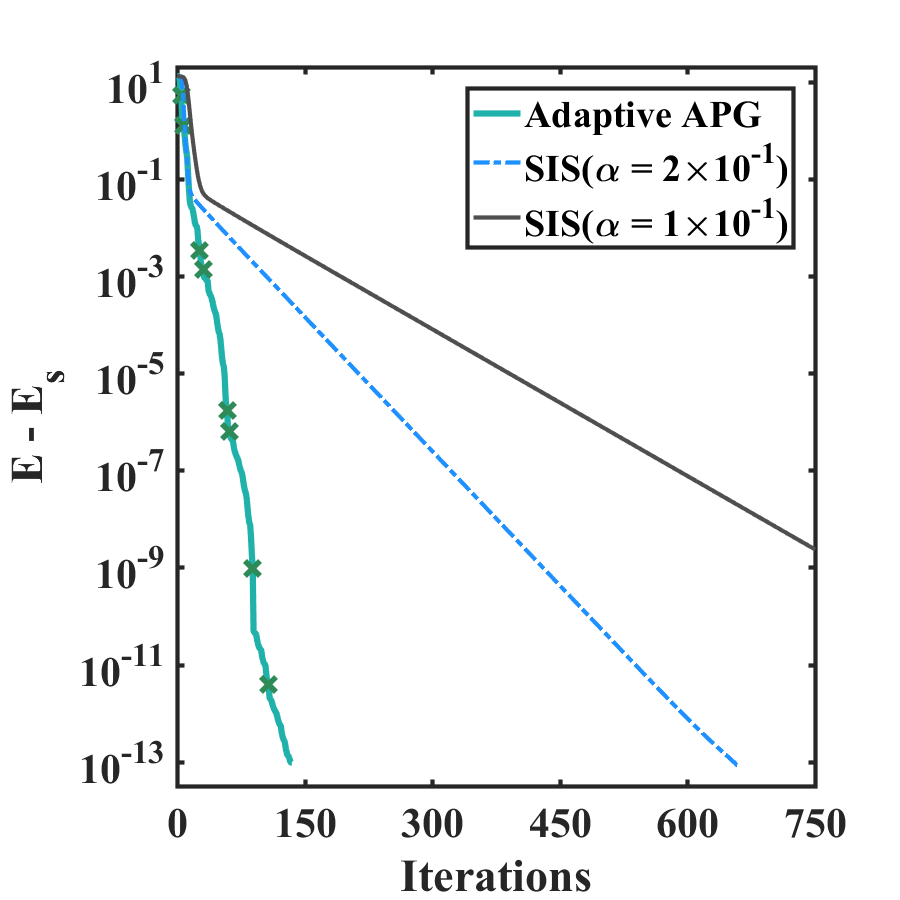}
	\includegraphics[scale=0.4]{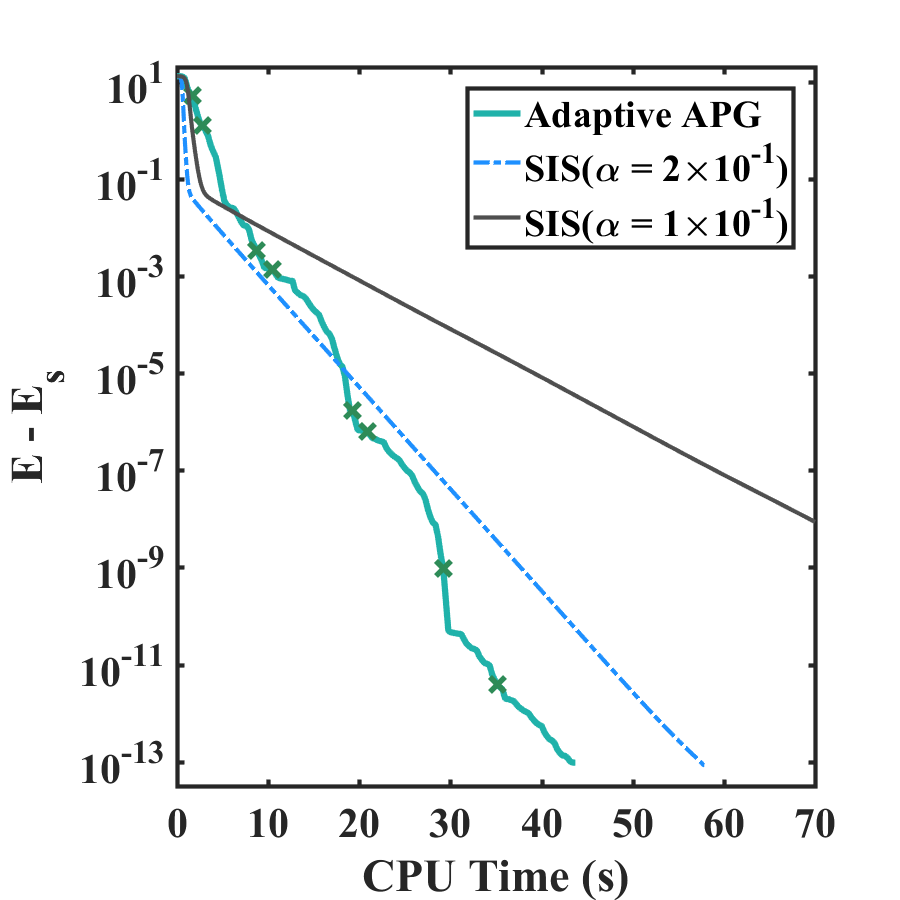}
	\caption{\label{fig:DG:comparison}
	Comparison of convergence across different algorithms for
	computing the double gyroid phase.
	The vertical axis is the difference between the energy value
	in current step and the lowest attained energy value. 
	On left, the horizontal axis is the number of
	iterations. On right, the horizontal axis is time taken.
	The $\times$s mark where restarts occurred.
	}
\end{figure}
\begin{figure}[htbp]
	\centering
	\includegraphics[scale=0.4]{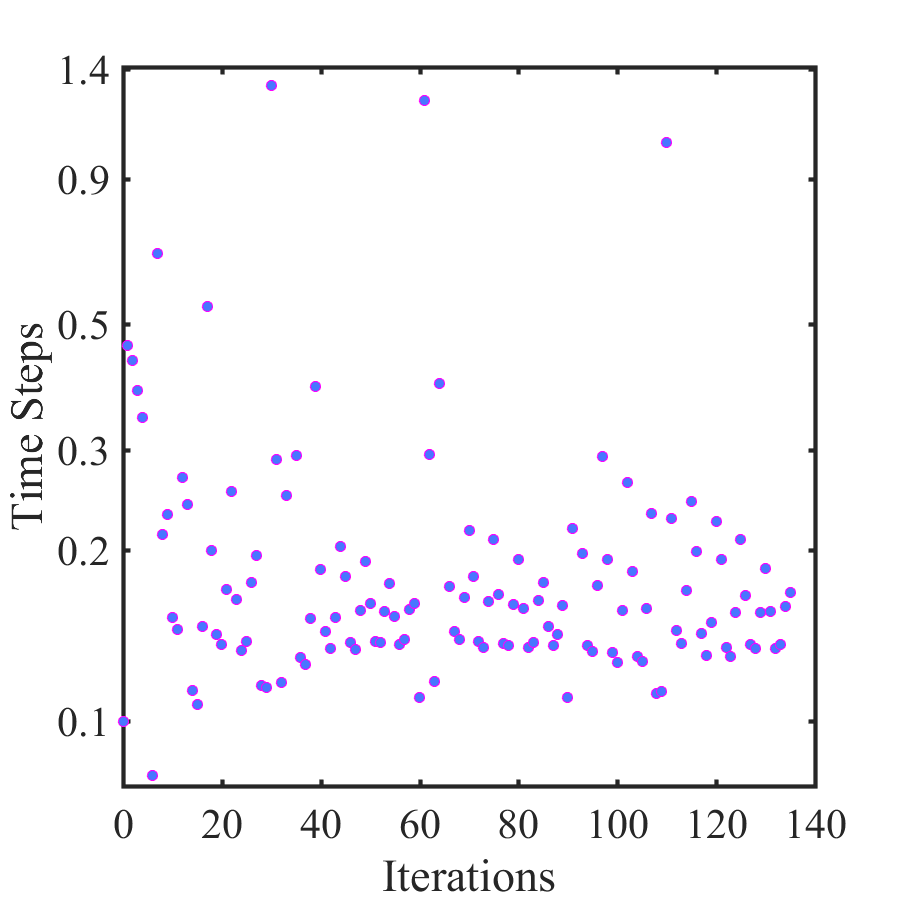}
	\caption{\label{fig:DG:step}
	The adaptive time steps obtained by the adaptive APG in
	computing the double gyroid phase.
	}
\end{figure}
In particular,
The adaptive APG needs $149$ steps to achieve the error level of
$10^{-13}$, while the SIS requires $660$ iterations for 
$ \alpha  = 0.2$ and $1190$ steps when $ \alpha  = 0.1$. 
Our proposed approach requires the linear search techniques to
obtain the adaptive time step length, it may spend more time than fixed
$ \alpha $ method in each iteration sometimes. However, due to
the adaptive strategy, our proposed approach still costs less
CPU time than the SIS. 

\subsubsection{Sigma phase}
\label{subsec:sigma}

The second periodic structure considered here is the sigma phase,
which is a complicated spherical packed phase recently discovered in
block copolymer systems\,\cite{lee2010discovery}. 
For such a pattern, we implement our algorithm 
on bounded computational domain $[0, 27.7884)\times [0,
27.7884)\times [0, 14.1514)$. The
initial values are obtained from \cite{xie2014sigma,
arora2016broadly}.
When computing the sigma phase, the parameters 
are set as $\xi = 1.0, \tau = 0.01, \gamma = 2.0$. 
The exact solution is obtained numerically by using 
$256\times 256 \times 128$ spatial discretization points whose
morphology is presented in Figure \ref{fig:sigma}.
Correspondingly, the convergent energy value is $E_s = -0.93081648457086$. 
As far as we know, it is the first time to find such 
complicated sigma phase in such a simple PFC model.

\begin{table}[htbp]
	\centering
	\caption{Accuracy of the Fourier pseudospectral discretization for
	the sigma phase in the LB model simulations.  
	The solution with $ 256 \times 256 \times 128 $ is used as reference solution.
	}
	\label{tab:sigma:accuracy}
	\begin{tabular}{|c|c|c|c|c|}
	\hline
	 & DOF & $ 128 \times 128 \times 64 $ & $ 160 \times 160 \times 80 $ & $ 200 \times 200 \times 100 $
	\\ \hline 
	Sigma   &
	\makecell{
	$\|\Phi_s - \Phi_h\|_2$
	\\
	$|E_s - E_h(\Phi_{s,h})|$
	}
	& 
	\makecell{ 2.2710e-06 \\ 4.2930e-03 
	}
	& 
	\makecell{ 7.1800e-11 \\ 2.3648e-14 
	}
	& 
	\makecell{ 7.3107e-12 \\ 2.3315e-15 
	}
	\\ \hline 
   \end{tabular}
\end{table}

\begin{figure}[htbp]
	\centering
	\includegraphics[scale=0.3]{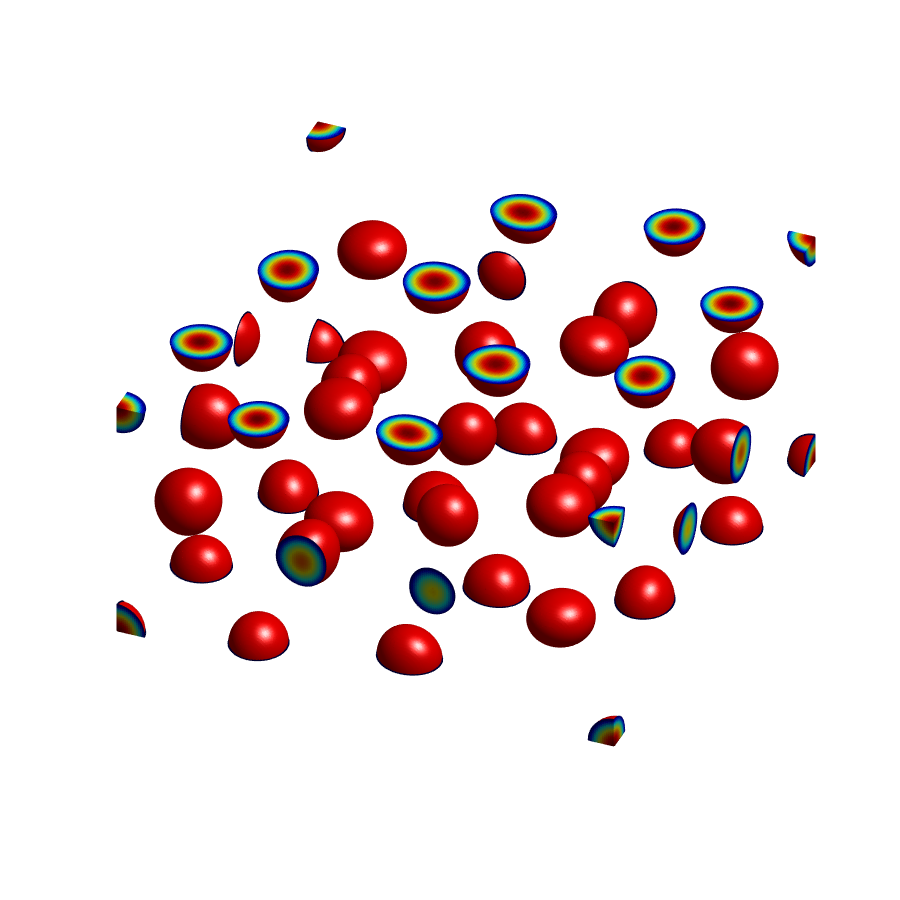}
	\hspace{1cm}
	\includegraphics[scale=0.28]{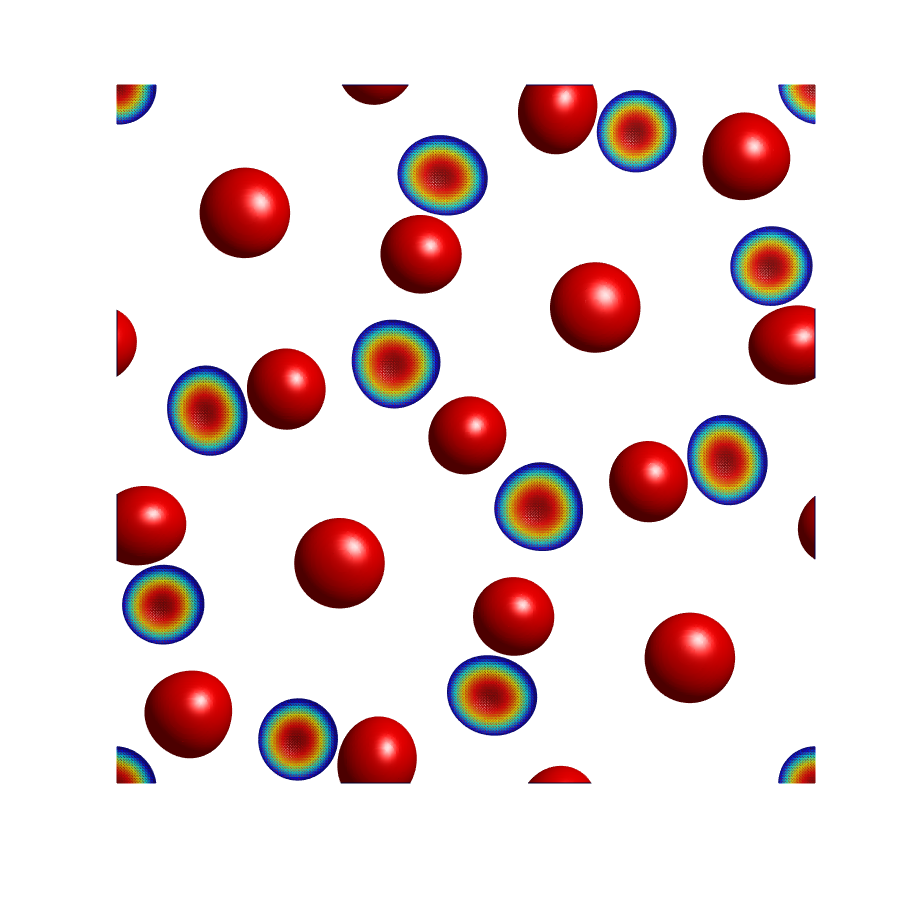}
	\caption{\label{fig:sigma}
	{\small The stationary sigma phase from two views in LB model 
	with $\xi = 1.0, \tau = 0.01, \gamma = 2.0$.}
	}
\end{figure}
Figure \ref{fig:sigma:comparison} gives the convergence between
the adaptive APG method and the SIS for the energy difference. Again, on
the premise of energy dissipation, the time step $ \alpha $ in
the SIS is chosen as $0.4$ to demonstrate its best
performance. Our proposed algorithm still can obtain adaptive
time step by the linear search technique, as demonstrated in
Figure \ref{fig:sigma:step}. Obviously, from Figure
\ref{fig:sigma:comparison}, the adaptive APG algorithm is more
efficient than the SIS. 
In concrete, the adaptive APG approach reaches an error about
$10^{-13}$ in $247.3$ secs with $174$ iterations.
The SIS with fixed step length $ \alpha =0.4$ ($0.3$)
requires $851$ ($1086$) iterations and $370.3$ ($474.5$) secs to
achieve the same accuracy.
\begin{figure}[htbp]
	\centering
	\includegraphics[scale=0.4]{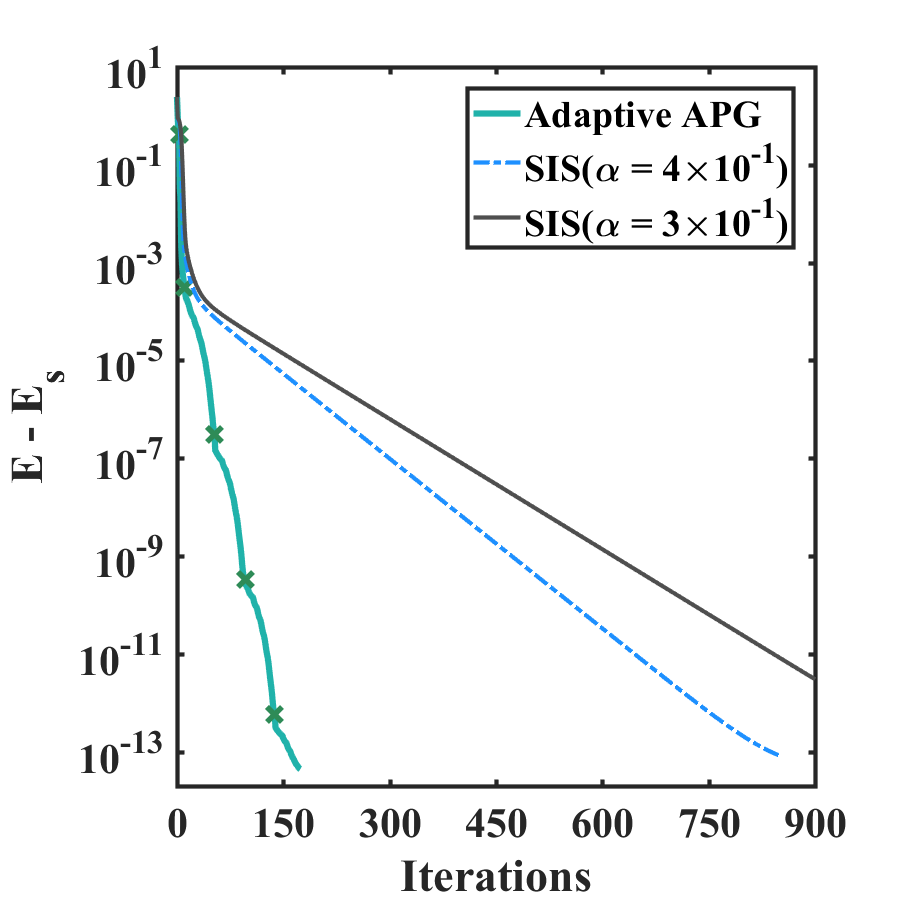}
	\includegraphics[scale=0.4]{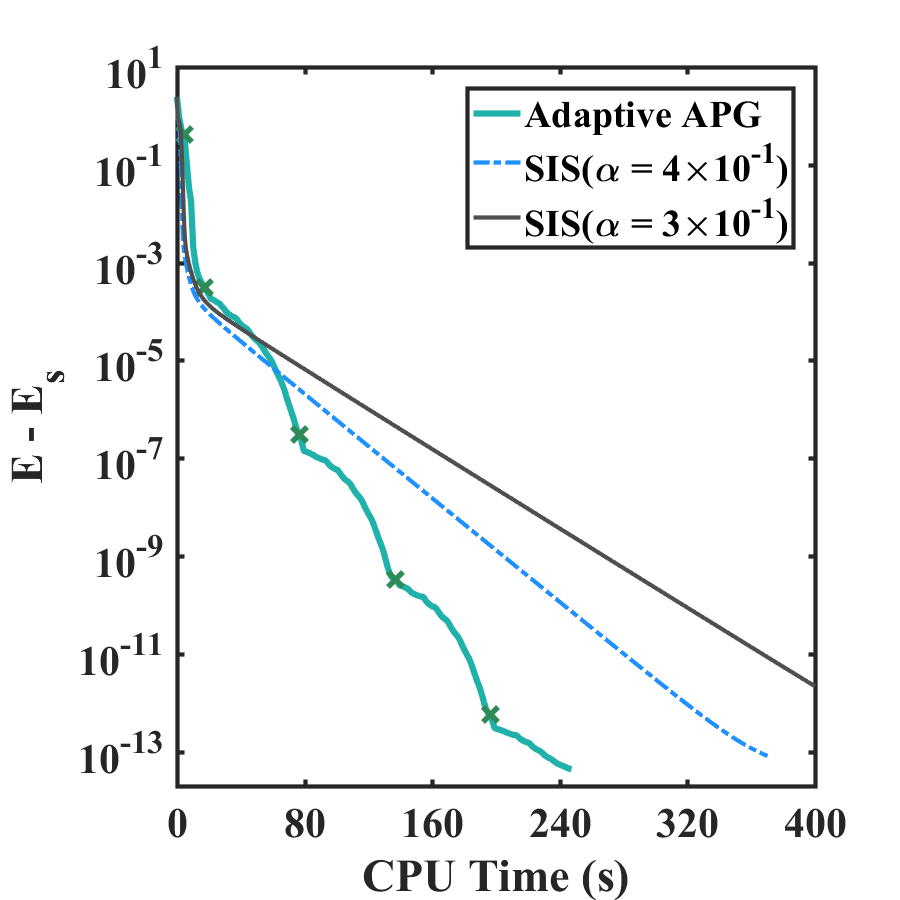}
	\caption{\label{fig:sigma:comparison}
	Comparison of convergence across different algorithms for
	computing the sigma phase.
	The vertical axis is the difference between the energy value
	in current step and the lowest attained energy value. 
	On left, the horizontal axis is the number of
	iterations. On right, the horizontal axis is time taken.
	The $\times$s mark where restarts occurred.
	}
\end{figure}
\begin{figure}[htbp]
	\centering
	\includegraphics[scale=0.4]{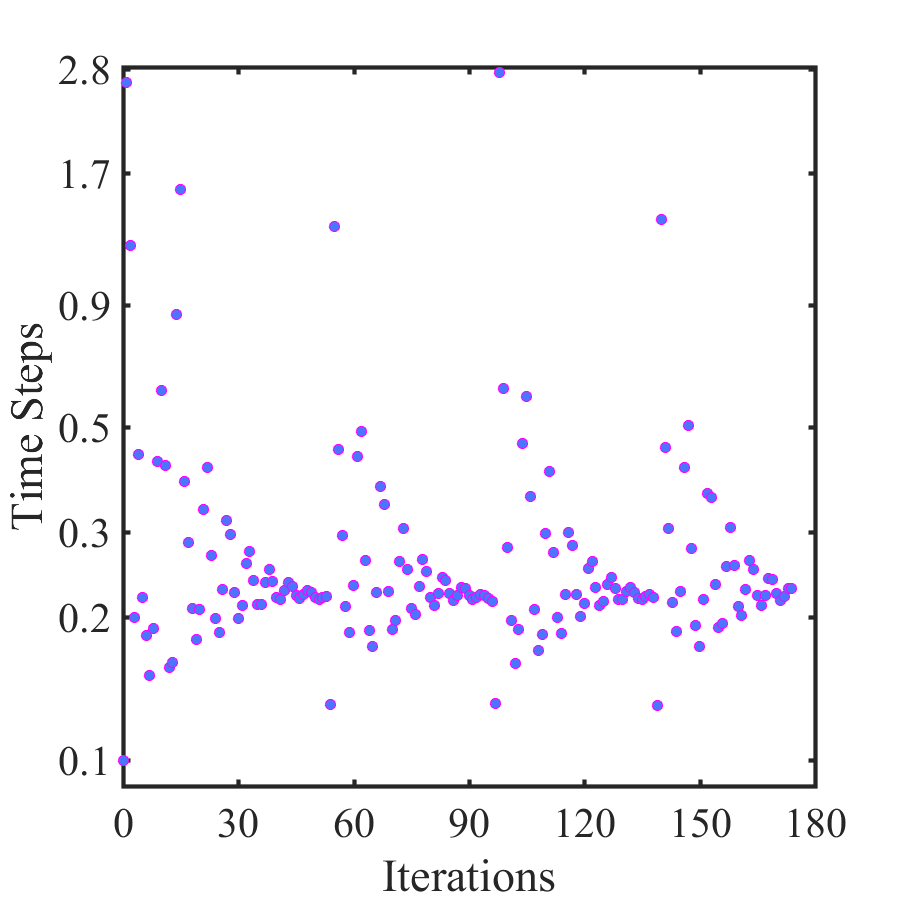}
	\caption{\label{fig:sigma:step}
	The adaptive time steps obtained by the adaptive APG in
	computing the sigma phase.
	}
\end{figure}

\subsection{Quasicrystals}
\label{subsec:qc}

For the LP free energy, we take the two dimensional dodecagonal
quasicrystal as an example to examine the performance of our
proposed approach. 
For dodecagonal quasicrystals, two length scales $q_1$
and $q_2$ equal to $1$ and $2\cos(\pi/12)$, respectively.
Two dimensional dodecagonal quasicrystals can be embedded into
four dimensional periodic structures, therefore, the projection
method is required to implement in four dimension.
The $4$-order invertible matrix $\bB$ associated with to
four dimensional periodic structure is chosen as $\bI_4$. 
The corresponding computational domain in real space is $[0,2\pi)^4$. 
The projection matrix $\calP$ in Eq.\,\eqref{eq:pm} of the
dodecagonal quasicrystals is
\begin{equation}
	\mathcal{P} =\left(
	\begin{array}{cccc}
		1 & \cos(\pi/6) & \cos(\pi/3) & 0 \\
		0 & \sin(\pi/6) & \sin(\pi/3) & 1 
	\end{array}
\right).
\label{eqn:DDQC:projMatrix}
\end{equation}
The initial solution is 
\begin{equation}
\begin{aligned}
	\phi(\br) = \sum_{\bh\in\Lambda^{QC}_0} \hphi(\bh)
    e^{i[(\mathcal{P}\cdot\mathbf{B}\bh)^{T}\cdot\br]},
    ~~\br\in\mathbb{R}^2,
\end{aligned}
\end{equation}
where initial lattice points set $\Lambda_0^{QC}\subset\bbZ^4$ on
which the Fourier coefficients $\hphi(\bh)$ located are nonzero of
dodecagonal quasicrystal is given in the Table \ref{tab:LP:initial}.
\begin{table}[!htbp]
\caption{\label{tab:LP:initial} The initial lattice points 
$\Lambda_0^{QC}$ of dodecagonal quasicrystals.}
\centering
\begin{tabular}{|c|c|}
\hline
$\bh \in \Lambda_0^{QC}$ &
\makecell{
	(0     1     0   -1)
	(0    -1     0    1)
	(1     0     0    0)
   (-1     0     0    0)
	(0     1     0    0)
	(0    -1     0    0)
	\\
	(0     0     1    0)
	(0     0    -1    0)
	(0     0     0    1)
	(0     0     0   -1)
   (-1     0     1    0)
	(1     0    -1    0)
	\\
	(1     1     0   -1)
   (-1    -1     0    1)
	(1     1     0    0)
   (-1    -1     0    0)
	(0     1     1    0)
	(0    -1    -1    0)
	\\
	(0     0     1    1)
	(0     0    -1   -1)
   (-1     0     1    1)
	(1     0    -1   -1)
   (-1    -1     1    1)
	(1     1    -1   -1)
	}
\\\hline
\end{tabular}
\end{table}

\begin{figure}[htbp]
	\centering
	\includegraphics[scale=0.4]{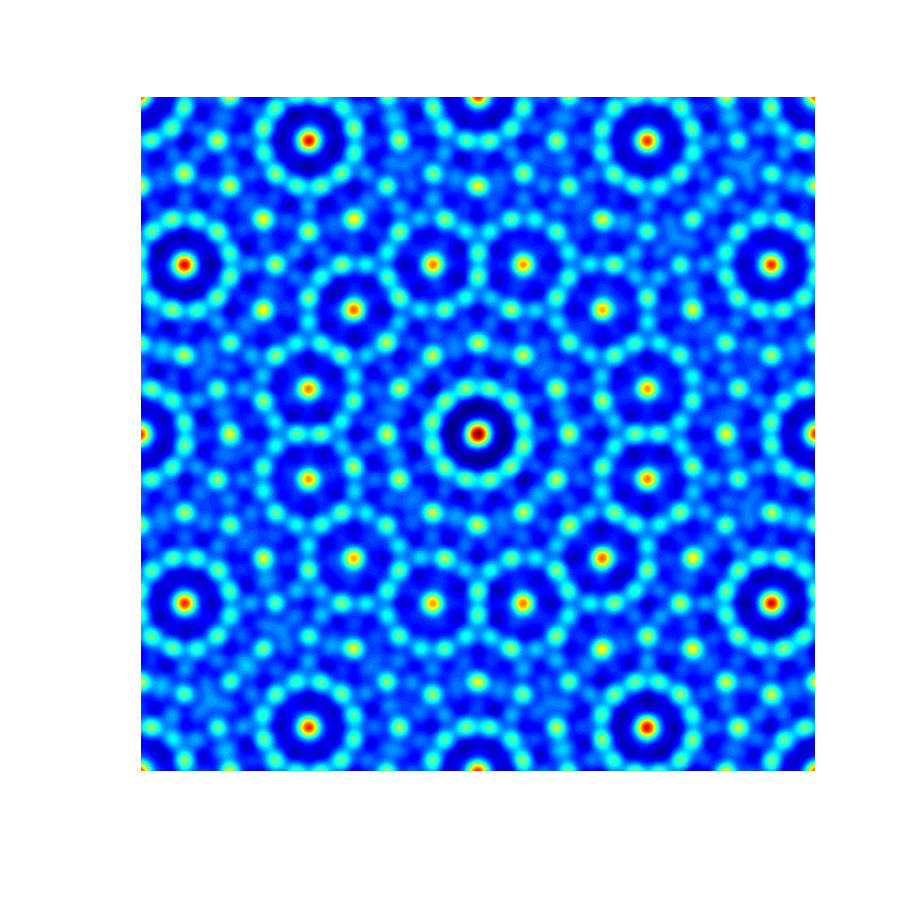}
	\includegraphics[scale=0.4]{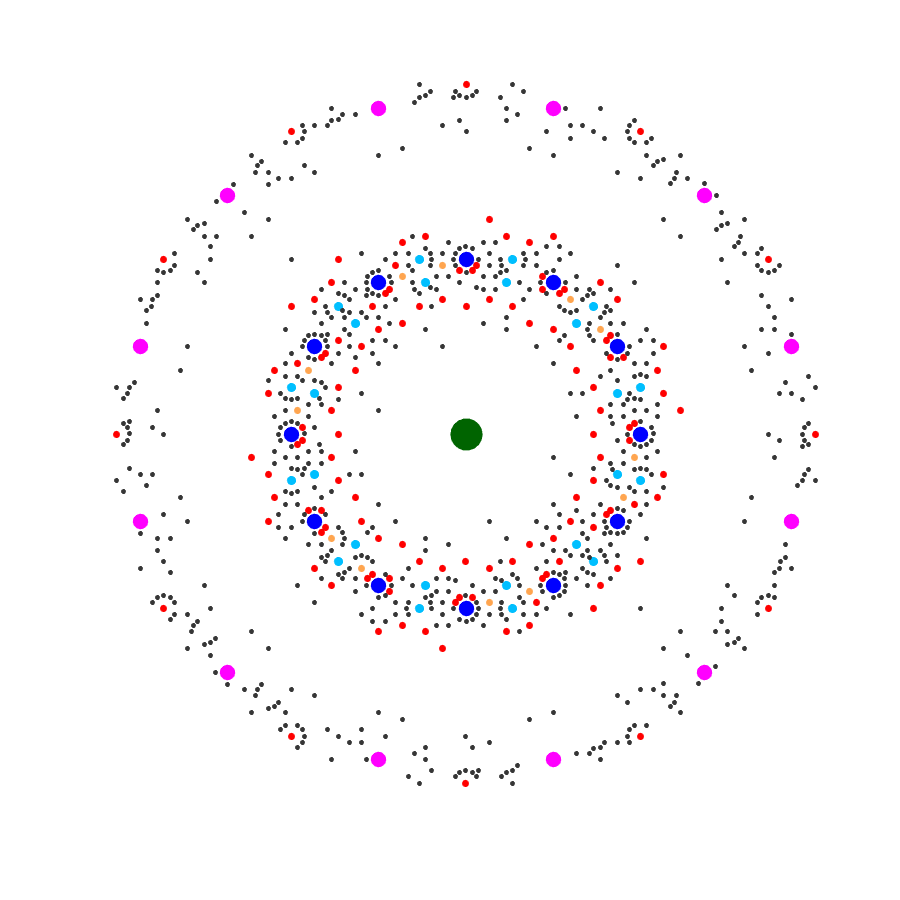}
	\caption{\label{fig:DDQC}
	The stationary dodecagonal quasicrystal in LP model 
	with $c = 1.5, \varepsilon = -6.0, \kappa = 0.3$. 
	The left plot is the physical morphology. 
	The right subfigure is the Fourier spectral
	points whose coefficient intensity is larger than $0.01$. 
	}
\end{figure}

When computing the dodecagonal quasicrystal,
the parameters in LP model are set as $c = 1.5, \varepsilon =
-6.0, \kappa = 0.3$, and $38^4$ spatial discretization points are used.
The convergent stationary quasicrystal including its morphology
and Fourier spectrum is given in Figure \ref{fig:DDQC}. The
finally convergent energy value obtained by the adaptive APG
approach is $E_s=-5.76164741513328$. The iterative behavior of
our proposed method and the SIS with different fixed time steps,
$ \alpha =0.1, 0.05, 0.005$, is found in
Figure \ref{fig:DDQC:comparison}. The adaptive time steps of our
proposed approach is given in Figre \ref{fig:DDQC:step}.
In the SIS, the energy difference decreases to the error level of about
$10^{-6.8}$, then increases for all given time step $ \alpha $.
However, the adaptive APG algorithm is always energy
dissipation as Theorem \ref{thm:convergence} predicted. These
results demonstrates that the adaptive APG approach is more robust for finding the
stationary states.

\begin{figure}[htbp]
	\centering
	\includegraphics[scale=0.4]{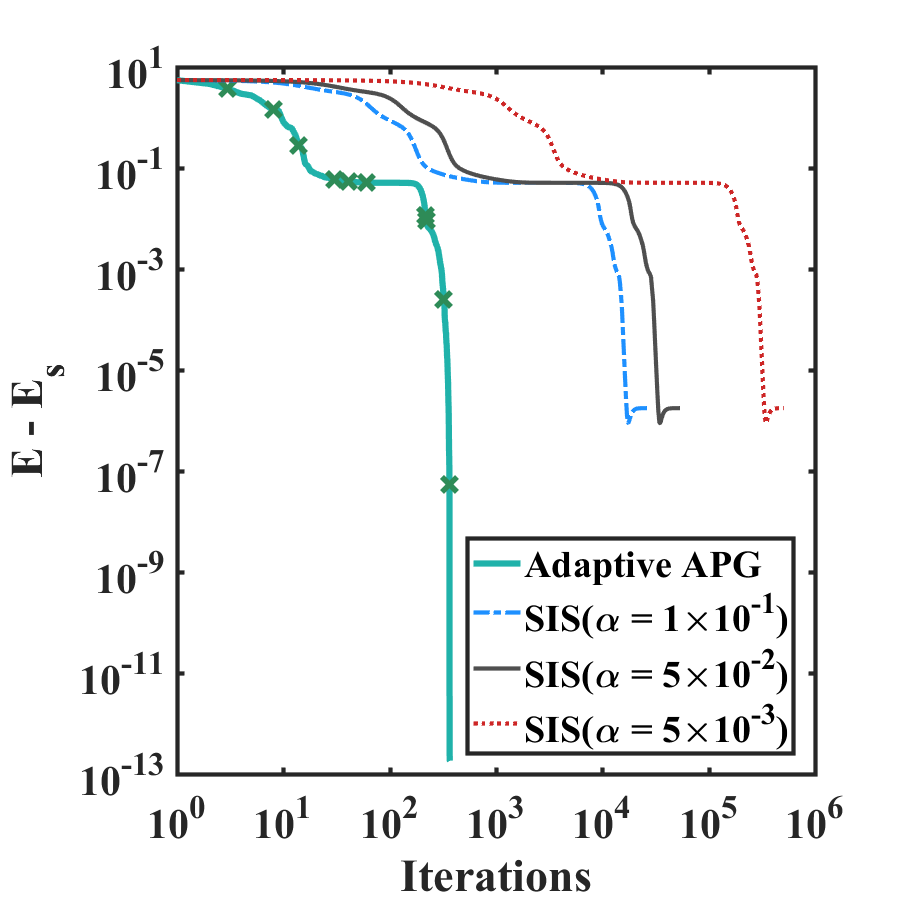}
	\includegraphics[scale=0.4]{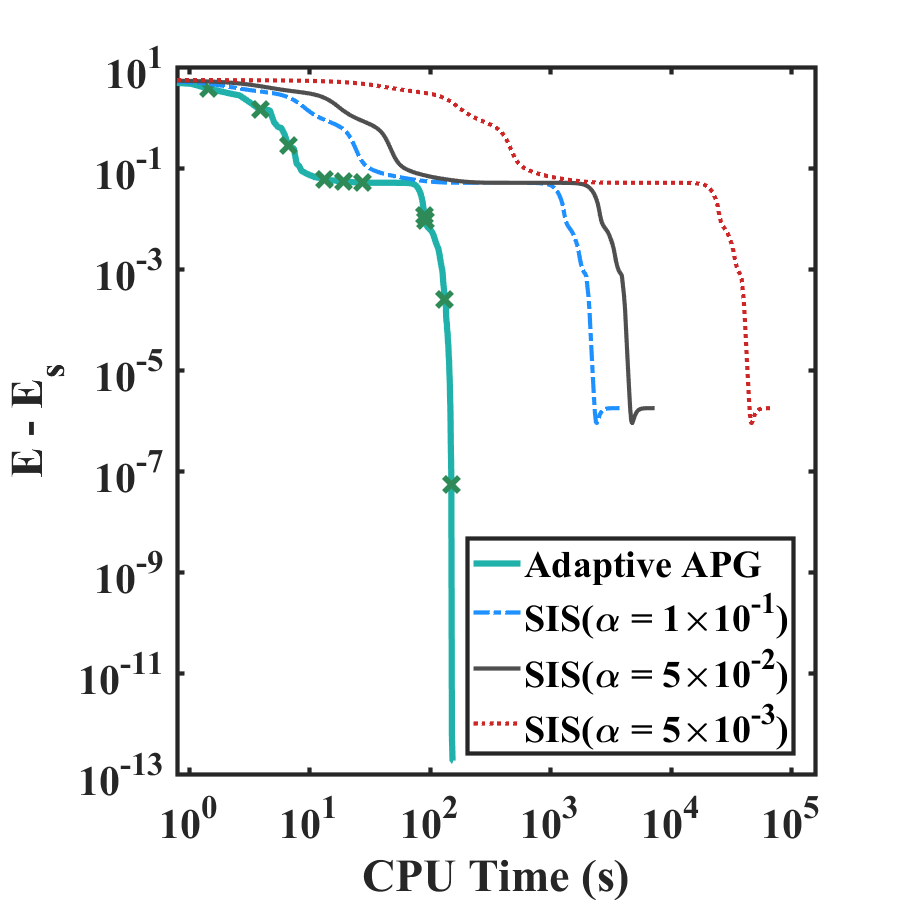}
	\caption{\label{fig:DDQC:comparison}
	Comparison of convergence across different algorithms for
	computing the dodecagonal phase.
	The vertical axis is the difference between the energy value
	in current step and the lowest attained energy value. 
	On left, the horizontal axis is the number of
	iterations. On right, the horizontal axis is time taken.
	The $\times$s mark where restarts occurred.
	}
\end{figure}
\begin{figure}[htbp]
	\centering
	\includegraphics[scale=0.4]{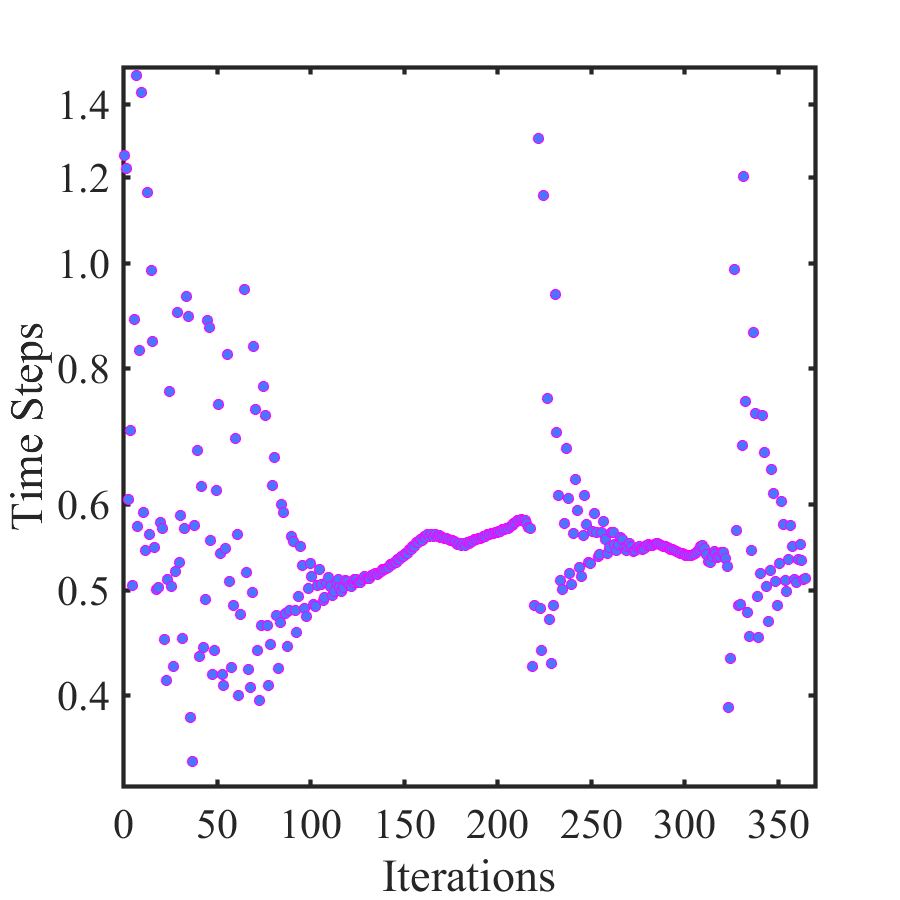}
	\caption{\label{fig:DDQC:step}
	The adaptive time steps obtained by the adaptive APG in
	computing the dodecagonal phase.
	}
\end{figure}

\section{Conclusion}
\label{sec:conclusion}
In this paper, a fast, efficient, and robust computational
approach has been proposed to find the stationary states of PFC
models. The adaptive APG method is obtained through a combination
of the SIS and the restart APG approach.
Instead of formulating the energy
minimization as a gradient flow, we applied the adaptive APG
method directly on the discretized energy with proved local
convergence. Extensive results in computing periodic crystals and
quasicrystals have shown its advantage in terms of
computation efficiency without loss of the accuracy. Moreover,
the preliminary connections between the numerical schemes in
solving gradient flow and the generalized proximal operator are
present in this work and motivate us to continue finding its deep
relationship in future.

\section*{Acknowledgements} 

This work is supported by the National Natural Science Foundation
of China (11771368) and the Project of Scientific Research
Fund of Hunan Provincial Science and Technology Department (2018WK4006).
KJ is partially supported by the Hunan Science Foundation of China
(2018JJ2376), the Youth Project Hunan Provincial Education Department
of China (Grant No. 16B257),

\end{document}